\documentclass[12pt,a4paper]{amsart}
\usepackage{mathtools}
\usepackage{xcolor}
\usepackage{amsmath,amstext,amssymb,amsthm,amsfonts}

\newcommand{\nc}{\newcommand}

\nc{\one}{\mbox{\bf 1}}
\nc{\invtensor}{\underset{\leftarrow}{\otimes}}
\nc{\const}{\operatorname{const}}

\nc{\ad}{\operatorname{ad}}
\nc{\tr}{\operatorname{tr}}

\nc{\tp}{\operatorname{top}}
\nc{\rank}{\operatorname{rank}}
\nc{\corank}{\operatorname{corank}}
\nc{\codim}{\operatorname{codim}}
\nc{\sdim}{\operatorname{sdim}}
\nc{\mult}{\operatorname{mult}}

\nc{\fspn}{\operatorname{span}}
\nc{\Sym}{\operatorname{Sym}}
\nc{\sym}{\operatorname{sym}}
\nc{\id}{\operatorname{id}}
\nc{\Id}{\operatorname{Id}}
\nc{\Ree}{\operatorname{Re}}

\nc{\htt}{\operatorname{ht}}

\nc{\str}{\operatorname{str}}

\nc{\Ker}{\operatorname{Ker}}
\nc{\rker}{\operatorname{rKer}}
\nc{\im}{\operatorname{Im}}
\nc{\osp}{\mathfrak{osp}}
\nc{\sgn}{\operatorname{sgn}}
\nc{\F}{\operatorname{F}}
\nc{\Mod}{\operatorname{Mod}}
\nc{\DS}{\operatorname{DS}}
\nc{\Soc}{\operatorname{Soc}}
\nc{\Inj}{\operatorname{Inj}}
\nc{\Hom}{\operatorname{Hom}}
\nc{\End}{\operatorname{End}}
\nc{\supp}{\operatorname{supp}}
\nc{\Card}{\operatorname{Card}}
\nc{\Ann}{\operatorname{Ann}}
\nc{\Ind}{\operatorname{Ind}}
\nc{\COind}{\operatorname{Coind}}
\nc{\wt}{\operatorname{wt}}
\nc{\ch}{\operatorname{ch}}
\nc{\sch}{\operatorname{sch}}

\nc{\Stab}{\operatorname{Stab}}
\nc{\Sch}{{\mathcal S}\mbox{\em ch}}
\nc{\Irr}{\operatorname{Irr}}
\nc{\fspec}{\operatorname{Spec}}
\nc{\Res}{\operatorname{Res}}
\nc{\Aut}{\operatorname{Aut}}
\nc{\Ext}{\operatorname{Ext}}

\nc{\Fract}{\operatorname{Fract}}
\nc{\gr}{\operatorname{gr}}

\nc{\deff}{\operatorname{def}}
\nc{\HC}{\operatorname{HC}}

\nc{\Snow}{\operatorname{Snow}}
\nc{\red}{\operatorname{red}}

\nc{\wdchi}{\widetilde{\chi}}
\nc{\wdH}{\widetilde{H}}
\nc{\wdN}{\widetilde{N}}
\nc{\wdM}{\widetilde{M}}
\nc{\wdO}{\widetilde{O}}
\nc{\wdR}{\widetilde{R}}

\nc{\wdV}{\widetilde{V}}

\nc{\wdC}{\widetilde{C}}

\nc{\Obj}{\operatorname{Obj}}
\nc{\Dglie}{\operatorname{{\mathcal D}glie}}
\nc{\Fin}{\operatorname{{\mathcal F}in}}

\nc{\Adm}{\operatorname{\mathcal{A}dm}}

\nc{\Sg}{{\cS(\fg)}}
\nc{\Shg}{{\cS(\fhg)}}
\nc{\Ug}{{\cU(\fg)}}
\nc{\Uhg}{{\cU(\fhg)}}
\nc{\Sh}{{\cS(\fh)}}
\nc{\Uh}{{\cU(\fh)}}
\nc{\Uhh}{{\cU(\fhh)}}
\nc{\Zg}{{{\mathcal{Z}}(\fg)}}

\nc{\Vir}{{\mathcal{V}ir}}
\nc{\NS}{{\mathcal{N}S}}

\nc{\tZg}{{\widetilde{\mathcal Z}({\mathfrak g})}}
\nc{\Zk}{{\mathcal Z}({\mathfrak k})}

\newcommand{\dD}{\mathcal{D}}

\nc{\Up}{{\mathcal U}({\mathfrak p})}
\nc{\Ah}{{\mathcal A}({\mathfrak h})}
\nc{\Ag}{{\mathcal A}({\mathfrak g})}
\nc{\Ap}{{\mathcal A}({\mathfrak p})}
\nc{\Zp}{{\mathcal Z}({\mathfrak p})}
\nc{\cR}{\mathcal R}
\nc{\cS}{\mathcal S}
\nc{\cT}{\mathcal{T}}
\nc{\cC}{\mathcal C}
\nc{\cA}{\mathcal A}
\nc{\cB}{\mathcal B}
\nc{\cU}{\mathcal U}
\nc{\cZ}{\mathcal Z}
\nc{\cM}{\mathcal M}
\nc{\cL}{\mathcal L}
\nc{\cF}{\mathcal F}
\nc{\fg}{\mathfrak g}

\nc{\cK}{\mathcal{K}}
\nc{\CO}{\mathcal O}
\nc{\CR}{\mathcal R}

\nc{\cW}{\mathcal{W}}
\nc{\bM}{\mathbf{M}}
\nc{\bL}{\mathbf{L}}
\nc{\bN}{\mathbf{N}}

\nc{\zq}{\mathpzc q}

\nc{\fo}{\mathfrak o}
\nc{\fl}{\mathfrak l}
\nc{\fn}{\mathfrak n}
\nc{\fm}{\mathfrak m}
\nc{\fp}{\mathfrak p}
\nc{\fh}{\mathfrak h}
\nc{\ft}{\mathfrak t}
\nc{\fk}{\mathfrak k}
\nc{\fb}{\mathfrak b}
\nc{\fs}{\mathfrak s}
\nc{\fB}{\mathfrak B}

\nc{\vareps}{\varepsilon}
\nc{\varesp}{\varepsilon}
\nc{\veps}{\varepsilon}

\nc{\fsl}{\mathfrak{sl}}
\nc{\fgl}{\mathfrak{gl}}
\nc{\fso}{\mathfrak{so}}
\nc{\fosp}{\mathfrak{osp}}
\nc{\fsp}{\mathfrak{sp}}
\nc{\fq}{\mathfrak q}
\nc{\fsq}{\mathfrak{sq}}
\nc{\fpsl}{\mathfrak{psl}}

\nc{\fhg}{\hat{\fg}}
\nc{\fhn}{\hat{\fn}}
\nc{\fhh}{\hat{\fh}}
\nc{\fhb}{\hat{\fb}}
\nc{\hrho}{\hat{\rho}}

\nc{\hsl}{\hat{\fsl}}
\nc{\fpo}{\mathfrak{po}}
\nc{\dirlim}{\underset{\rightarrow}{\lim}\,}
\nc{\nen}{\newenvironment}
\nc{\ol}{\overline}
\nc{\ul}{\underline}
\nc{\ra}{\rightarrow}
\nc{\lra}{\longrightarrow}
\nc{\Lra}{\Longrightarrow}
\nc{\bo}{\bar{1}}
\nc{\Lla}{\Longleftarrow}

\nc{\Llra}{\Longleftrightarrow}

\nc{\thla}{\twoheadleftarrow}

\nc{\lang}{(}
\nc{\rang}{)}

\nc{\hra}{\hookrightarrow}

\nc{\iso}{\overset{\sim}{\lra}}

\nc{\ssubset}{\underset{\not=}{\subset}}

\nc{\vac}{|0\rangle}

\nc{\Thm}[1]{Theorem~\ref{#1}}
\nc{\Prop}[1]{Proposition~\ref{#1}}
\nc{\Lem}[1]{Lemma~\ref{#1}}
\nc{\Cor}[1]{Corollary~\ref{#1}}
\nc{\Conj}[1]{Conjecture~\ref{#1}}
\nc{\Claim}[1]{Claim~\ref{#1}}
\nc{\Defn}[1]{Definition~\ref{#1}}
\nc{\Exa}[1]{Example~\ref{#1}}
\nc{\Rem}[1]{Remark~\ref{#1}}
\nc{\Note}[1]{Note~\ref{#1}}
\nc{\Quest}[1]{Question~\ref{#1}}
\nc{\Hyp}[1]{Hypoth\`ese~\ref{#1}}
\nen{thm}[1]{\label{#1}{\bf Theorem.\ } \em}{}
\nen{prop}[1]{\label{#1}{\bf Proposition.\ } \em}{}
\nen{lem}[1]{\label{#1}{\bf Lemma.\ } \em}{}
\nen{cor}[1]{\label{#1}{\bf Corollary.\ } \em}{}
\nen{conj}[1]{\label{#1}{\bf Conjecture.\ } \em}{}

\nen{claim}[1]{\label{#1}{\bf Claim.\ } \em}{}

\nen{defn}[1]{\label{#1}{\bf Definition.\ } }{}
\nen{exa}[1]{\label{#1}{\bf Example.\ } }{}
\nen{rem}[1]{\label{#1}{\em Remark.\ } }{}
\nen{note}[1]{\label{#1}{\em Note.\ } }{}
\nen{exer}[1]{\label{#1}{\em Exercise.\ } }{}
\nen{sket}[1]{\label{#1}{\em Sketch of proof.\ } }{}
\nen{quest}[1]{\label{#1}{\bf Question.\ } \em}{}

\nen{hyp}[1]{\label{#1}{\bf Hypoth\`ese.\ } \em}{}
\begin{document}
\setcounter{section}{-1}
\setcounter{tocdepth}{1}

\title[Snowflake modules]{Snowflake modules and  Enright functor
for Kac-Moody superalgebras}
\author{Maria Gorelik,  Vera Serganova }

\address[]{Dept. of Mathematics, The Weizmann Institute of Science,Rehovot 7610001, Israel}
\email{maria.gorelik@weizmann.ac.il}

\address[]{Dept. of Mathematics,
University of California at Berkeley, Berkeley CA 94720}
\email{serganov@math.berkeley.edu}

\thanks{Supported in part by BSF Grant 2012227. The second author was
  partially supported by NSF grant DMS-1701532.}

\begin{abstract} We introduce a  class of
modules over Kac-Moody superalgebras; we call these modules
  snowflake. These modules are characterized by invariance property
of their characters with respect to a certain subgroup of the Weyl
group.  Examples of snowflake modules appear as admissible modules
in representation theory
of affine vertex algebras and in classification of bounded weight
modules. Using these modules we prove Arakawa's Theorem for the Lie superalgebra
$\mathfrak{osp}(1|2\ell)^{(1)}$.

\end{abstract}
\maketitle

\section{Introduction}
The main goal of this paper is introduction of certain class of
modules over Kac-Moody superalgebras; we call these modules {\em
  snowflake}. These modules are characterized by invariance property
of their characters with respect to a suitable subgroup of the Weyl
group. Examples of snowflake modules appear as ``admissible modules''
in representation theory
of affine vertex algebras and in classification of bounded weight
modules over $\mathfrak{osp}(m|2n)$~\cite{GG}.
According to Arakawa's Theorem~\cite{AM},\cite{A}
if $k$ is an admissible level, then
the simple vertex affine algebra $V_k(\fg)$
is ``rational in the category $\CO$''; in this case
all $V_k(\fg)$-modules in the category $\CO$
are certain snowflake modules. One of the results of our paper is
a proof of Arakawa's Theorem for the Lie superalgebra
$\mathfrak{osp}(1|2\ell)^{(1)}$.

Recall that the Kac-Moody superalgebras $\fg(A)$ with an
indecomposable Cartan matrix $A$
can be divided in two classes:
those, where $a_{ii}=2$ for each $i$ (we call them
{\em non-isotropic type}) and others (which we call
{\em isotropic type}). The main difference between these two classes
is that the superalgebras of isotropic type
have several Cartan matrices and the Weyl group
should be extended to the Weyl groupoid using so-called odd
reflections. In~\cite{H} the superalgebras of
isotropic type are classified; in the symmetrizable case
these algebras are either finite-dimensional or (twisted)
affinizations of finite-dimensional superalgebras.
In Section~\ref{sect1} we recall the construction
of Kac-Moody superalgebra
and results of~\cite{DGK} and~\cite{KK}, which provide a description of simple modules
in the blocks of the category $\CO$ in the symmetrzable case (see~\ref{blocks}).

In this paper we study snowflake modules for symmetrizable
non-isotropic type superalgebras. In this case we prove
an analogue of Weyl-Kac character formula for non-critical
simple snowflake modules and show that the corresponding category
is semisimple.

Our main tool is the Enright functor introduced in~\cite{En}
and adapted to superalgebra setting in~\cite{IK}. Our approach to Enright functor is similar to~\cite{D};
another approach is developed in~\cite{J}.
We consider an arbitrary superalgebra $\fg$ containing
$\fsl_2$ which acts locally finitely on $\fg$. We fix the standard basis
$\{e,h,f\}$ of $\fsl_2$ and denote by $\cM(a)^e$ the category of
$\fg$-modules on which $e$ acts locally finitely and $h$ acts
diagonally with eigenvalues in $a+\mathbb Z$. The Enright functor
$\cC_a:\cM(a)^e\to \cM(-a)^e$ is the composition of twisted
localization with respect to $f$ and the Zuckerman functor which maps
a module to its  $e$-locally finite part. (In this way we avoid
superbinomial coefficients introduced in~\cite{IK}.) The Enright
functor commutes with the restriction to any subalgebra of $\fg$
containing $\fsl_2$. If $a\not\in\mathbb Z$, then $\cC_a$ is an
equivalence of categories. Moreover, in this case the character of $N$
and $\cC_a(N)$ are related by a simple formula, see~\Thm{thmchar}.
Let us note that all our results hold for supercharacters, see~\ref{Nnu}.

In the last section we discuss snowflake modules for isotropic type
superalgebras.
We hope that snowflake modules might provide a suitable framework for generalization
of
Arakawa's Theorem for these superalgebras.

The authors are grateful to T.~Arakawa, J.~Bernstein, A.~Joseph
and V.~Kac for helpful discussions.

\section{Kac--Moody Lie superalgebras}\label{sect1}
\subsection{Construction}\label{contra}
We start from the following data: an $\ell\times\ell$  complex matrix $A$
and a parity map $p:\{1,\ldots,\ell\}\to\mathbb{Z}_2$ with the following condition:

(A00) $a_{ij}=0$ implies $a_{ji}=0$;

(A0) if $p(i)=0$, then $a_{ii}=2$ and $a_{ij}\in\mathbb{Z}_{\leq 0}$ for $j\not=i$;

(A1) if $p(i)=1$, then either $a_{ii}=2$ and $a_{ij}\in
2\mathbb{Z}_{\leq 0}$ for $j\not=i$
or $a_{ii}=0$.

Let $(\fh,\Sigma,\Sigma^{\vee})$ be the realization of $A$ in the
sense of~\cite{Kbook2}, Ch. 1: $\fh$ is an even vector space with
$\dim\fh=2\ell- rank A$,
$\Sigma=\{\alpha_i\}_{i=1}^{\ell}\subset \fh^*$
and $ \Sigma^{\vee}=\{\alpha_i^{\vee}\}_{i=1}^{\ell}\subset\fh$
are linearly independent sets satisfying $\alpha_j(\alpha_i^{\vee})=a_{ij}$
for $i,j=1,\ldots,\ell$.
We set $p(\alpha_i):=p(i)$  and
extend $p:\Sigma \to\mathbb{Z}_2$ to $p:\mathbb{Z}\Sigma\to\mathbb{Z}_2$
by linearity; we  call $\nu \in \mathbb{Z}\Sigma$ even (resp., odd)
if $p(\nu)=0$ (resp., $p(\nu)=1$).

For each $i$ such that $p(i)=1$ and $a_{ii}=0$ we define an
{\it odd reflection} of our data: a  new set
$\Sigma'=r_{\alpha_i}(\Sigma):=\{\alpha'_1,\dots,\alpha_n'\}$ by the formula
$$\alpha_j'=\begin{cases}\alpha_j\,\text{if}\, j\neq i,\,a_{ij}=0\\
  \alpha_i+\alpha_j\,\text{if}\, j\neq i,\,a_{ij}\neq 0\\ -\alpha_j\,\text{if}\, j= i,\end{cases}$$
 a new set $(\Sigma')^\vee$ by
 $$(\alpha_j')^\vee=\begin{cases}\alpha_j^\vee\,\text{if}\,\,a_{ij}=0\\
   a_{ij}\alpha_j^\vee+a_{ji}\alpha_i^\vee\,\text{if}\,\,a_{ij}\neq 0,\,a_{jj}=2,\,\,a_{ji}=-1\\
    \frac{a_{ij}\alpha_j^\vee+a_{ji}\alpha_i^\vee}{a_{ij}a_{ji}}\,\text{if}\,\,a_{ij}\neq 0,\,a_{jj}=0\\
   \frac{a_{ij}\alpha_j^\vee+a_{ji}\alpha_i^\vee}{a_{ij}(a_{ji}+1)}\,\text{if}\,\,a_{ij}\neq 0,\,a_{jj}=2,\,a_{ji}\neq -1,\end{cases}$$
a new matrix $A'=r_i(A)$ by $a_{jk}' =\alpha_k'((\alpha_j')^{\vee})$,
and a new parity map
$p'$ by
$$p'(j)=\begin{cases}p(j)\,\,\text{if}\,\, a_{ij}=0\\ p'(j)+1\,\,\text{if}\,\,   a_{ij}\not=0.\end{cases}$$

We say that $(A,p)$ is a {\em Cartan supermatrix} if all pairs obtained from $(A,p)$
by chains of odd reflections satisfy (A00), (A0), (A1). In what follows we always assume that
$(A,p)$ is a Cartan supermatrix. Two supermatrices $(A,p)$ and $(A',p')$ are equivalent if one is
obtained from another by a chain of odd reflections.

Let $\tilde{\fg}(A)=\tilde{\fn}_+\oplus\fh\oplus\tilde{\fn}_-$ be a Lie superalgebra
generated by $\fh, e_i,f_i$ with $i=1,\ldots,\ell$ with $p(e_i)=p(f_i)=p(i)$
subject to the relations
$$[e_i,f_j]=\delta_{ij}\alpha_i^{\vee},\  [h,e_i]=\alpha_i(h)e_i,\
[h,f_i]=-\alpha_i(h)f_i$$
and $[e_i,e_i]=[f_i,f_i]=0$ if $p(i)=1$ with $a_{ii}=0$. We set $e_{\alpha_i}:=e_i$
and $f_{\alpha_i}:=f_i$.

A {\it Kac--Moody superalgebra} $\fg(A)$ is the quotient of
$\tilde{\fg}(A)$
by the maximal ideal which intersects $\fh$ trivially.
We identify $\fh$ with its image in $\fg(A)$ and call $\fh$ the Cartan subalgebra of $\fg(A)$.

Denote by $\cB$ the set of all bases obtained from $\Sigma$ by
odd reflections.  If $\Sigma'\in\cB$ and $A'$ is the corresponding matrix, then
$\fg(A)$ is isomorphic to $\fg(A')$, \cite{S3}.
(In other words, equivalent supermatrices define isomorphic
superalgebras.)

We denote by $\Delta$ the set of roots of $\fg(A)$ and by
$\Delta_{\ol{0}}$ (resp., by $\Delta_{\ol{1}}$) the set of even (resp., odd) roots.
{For $\Sigma\in\cB$ and $\alpha\in\Sigma$ with $a_{\alpha\alpha}=0$
we denote by $r_{\alpha}\Sigma$ a new base obtained from $\Sigma$
by the odd reflection with respect to $\alpha$ ($r_{\alpha}\Sigma\subset \Delta$).}
If $\Delta^+(\Sigma')$ is the set of positive roots for $\Sigma'\in\cB$,
then
$$\Delta^+(r_{\alpha}\Sigma)=(\Delta^+(\Sigma)\setminus\{\alpha\})\cup\{-\alpha\},\ \ \ \
\Delta^+(\Sigma')\cap \Delta_{\ol{0}}=\Delta^+(\Sigma)\cap \Delta_{\ol{0}}.$$

\subsubsection{}
For a base $\Sigma$ we write $\Sigma=\Sigma_1\coprod\Sigma_2$ if $a_{ij}=a_{ji}=0$ for every pair
$\alpha_i\in\Sigma_1$, $\alpha_j\in\Sigma_2$; in this case
$\fg(A)=\fg(A_1)\times\fg(A_2)$ and $\Delta=\Delta_1\coprod \Delta_2$,
where $A_1,A_2$ are Cartan matrices corresponding to $\Sigma_1,\Sigma_2$.
A base $\Sigma$ is called connected if  $\Sigma$ can not
be decomposed as $\Sigma_1\coprod\Sigma_2$ with the above condition.
A subset $\Sigma'\subset\Sigma$ is called a connected component if
$\Sigma'$ is connected and $\Sigma=\Sigma'\coprod (\Sigma\setminus\Sigma')$
(i.e. $a_{ij}=a_{ji}=0$ for every pair
$\alpha_i\in\Sigma'$, $\alpha_j\not\in\Sigma'$). If $\Sigma$ is
connected, then any base obtained
from $\Sigma$ by odd reflections
is connected; in this case we call
the Cartan supermatrix $(A,p)$ and
the corresponding algebra $\fg(A)$ {\em indecomposable}.

\subsection{The Lie algebra $\fg_{pr}$}\label{fgB}
A root $\alpha\in\Delta_{\ol{0}}$ is called {\em principal} if
$\alpha\in\Sigma$ or $\alpha/2\in\Sigma$ for some $\Sigma\in\cB$.
For $\alpha\in \Pi_{pr}$ the root spaces $\fg_{\pm \alpha}$ are one-dimensional
and they generate  a subalgebra $\fsl_2$; in particular,
$\alpha^{\vee}$ is well-defined (it is independent on the choice
of $\Sigma'$ containing  $\alpha$).

Consider the subalgebra $\fg_{pr}$ generated by $\fg_{\pm\beta}$ with
$\beta\in\Pi_{pr}$ and $\fh$. Clearly, $\fg_{pr}\subset\fg_{\ol{0}}$.
By~\cite{S3}, Lem. 3.7, there exists a  homomorphism
$\fg_{pr}\to \fg(B)/\mathfrak{c}$, where $\fg(B)$ is the Kac-Moody algebra
corresponding to the base $\Pi_{pr}$ (with the Cartan matrix given by
$b_{ij}:=\beta_j(\beta^{\vee}_i)$) and $\mathfrak{c}$ is a subspace of the centre
of $[\fg(B),\fg(B)]$.

\subsubsection{Example: $\mathfrak{osp}(2|4)^{(2)}$}\label{CDbad}
Consider $\fg=\mathfrak{osp}(2|4)^{(2)}$ with the Dynkin diagram.
$$
\overset{\delta-\vareps_{1}}{\bullet}\Longleftarrow
\overset{\vareps_{1}-\vareps_2}{\circ}
\Longrightarrow\overset{\vareps_2}{\bullet}$$

The principal roots are
$\{2\delta-2\vareps_{1},\vareps_{1}-\vareps_2,2\vareps_2\}$
and the Lie algebra $\fg_{pr}$  is isomorphic to
$\mathfrak{sp}_4^{(1)}$ with the Dynkin diagram
$$
\overset{2\delta-2\vareps_{1}}{\circ}\Longrightarrow
\overset{\vareps_{1}-\vareps_2}{\circ}
\Longleftarrow\overset{2\vareps_2}{\circ}.$$

The minimal imaginary root of $\fg_{pr}$ is $2\delta$, but
 $\delta\in\Delta_{\ol{0}}$, so
$\fg_{pr}\not=\fg_{\ol{0}}$.

\subsubsection{Partial order}\label{parord}
We set  $\Delta^{++}:=\Delta\cap \mathbb{Q}_{\geq 0}\Pi_{pr}$
and introduce {  a partial order on $\fh^*$ by
$$\nu\geq \mu\ \text{ if } (\nu-\mu)\in\mathbb{Z}_{\geq 0}\Delta^{++}.$$
Note that $\nu\geq 0$ implies $\nu\in \mathbb{Z}_{\geq 0}\Sigma$
for  every $\Sigma\in\cB$.}

\subsection{Weyl group}\label{Weylgroup}
For $\alpha\in\Pi_{pr}$  we  introduce
$r_{\alpha}\in GL(\fh^*),\ r_{\alpha}^{\vee}\in GL(\fh)$ by the formulae
$$r_{\alpha}(\lambda):=\lambda-\lambda(\alpha^{\vee})\alpha,\ \
\  r_{\alpha}^{\vee}(h):=h-\alpha(h){\alpha}^{\vee}.$$
Then $r_{\alpha}^2=Id$, $r_{\alpha}\Delta_{\ol{i}}=\Delta_{\ol{i}}$
for $=0,1$ and
$\dim\fg_{\gamma}=\dim \fg_{r_{\alpha}\gamma}$.

We denote by $W$ (resp., by $W^{\vee}$) the Weyl group of
$\fg(A)$, which is the subgroup of $GL(\fh^*)$ (resp., of
$GL(\fh)$)
generated by the $r_{\alpha}$ (reps., by $r_{\alpha}^{\vee}$) for
$\alpha\in\Pi_{pr}$.  The homomorphism $\fg_{pr}\to
\fg(B)/\mathfrak{c}$ (see~\ref{fgB}) induces a surjective homomorphism
$W\to W(\fg(B))$; since the generators of $W$ satisfy
the Coxeter relations and $W(\fg(B))$ is a Coxeter group
(see~\cite{Kbook2}, 3.13)
this homomorphism is an isomorphism, i.e., $W\cong W(\fg(B))$ is a Coxeter group.
Since
$r_{\alpha}\lambda(h)=\lambda(r_{\alpha}^{\vee}h)$, the correspondence
$r_{\alpha}\mapsto r_{\alpha}^{\vee}$ gives a group isomorphism $W\iso W^{\vee}$
and
$$\forall \lambda\in\fh^*,h\in\fh\ \  (w^{-1}\lambda)(h)=\lambda(w^{\vee}h).$$
We identify $W$ and $W^{\vee}$.
A root $\alpha$ is called {\it real} if there exist $\Sigma\in\cB$ and $w\in W$ such that $w\alpha$ or $w\alpha/2$ lies in $\Sigma$. Otherwise a root is called
  {\it imaginary}. We denote by $\Delta_{re}$ (resp., $\Delta_{im}$) the set of real (resp., imaginary) roots. By $\ol{\Delta}_{re}$ we denote the set of all
  real roots $\alpha$ such that $\alpha/2$ is not a root. Note that
  $$\ol{\Delta}_{re}:=W(\displaystyle\cup_{\Sigma\in\cB} \Sigma),\quad \Delta_{re}:=\ol{\Delta}_{re}\cup W\Pi_{pr}.$$

\subsubsection{}\label{wbetavee}
By~\cite{Kbook2}, Lem. 3.8,  for $\alpha\in\Pi_{pr}$ there exists
$r_{\alpha}^{ad}\in Aut \fg$ satisfying $r_{\alpha}^{\ad}|_{\fh}=r_{\alpha}^{\vee}$
and $r_{\alpha}^{ad}\fg_{\beta}=\fg_{r_{\alpha}\beta}$ for every
$\beta\in\Delta$.

Take $\alpha\in {\Delta}_{re}$. The roots spaces $\fg_{\pm\alpha}$
are one-dimensional
and the subalgebra generated by these roots spaces is
$\fsl_2$,  $\mathfrak{osp}(1|2)$ or $\fsl(1|1)$; in the first
two cases
$\alpha^{\vee}$ is well-defined; in the last case $\alpha^{\vee}$ is defined
up to a non-zero scalar. Note that $\fg_{\pm\alpha}$
 act locally nilpotently in the adjoint representation. We say
that $\alpha$ is {\em isotropic} (resp., {\em non-isotropic}) if $\alpha(\alpha^{\vee})=0$
(resp., $\alpha(\alpha^{\vee})=2$); $\alpha$ is isotropic if and only
if  $\fg_{\pm\alpha}$ generate $\fsl(1|1)$.
For $w\in W,\alpha\in {\Delta}_{re}$ one has
$(w\alpha)^{\vee}=w\alpha^{\vee}$.
For $c\in\mathbb{C}^*,\alpha\in {\Delta}_{re}$ we set $(c\alpha)^{\vee}:=c^{-1}\alpha^{\vee}$.

\subsubsection{}
We will use the following standard lemma {(see~\ref{parord} for notations)}.

\begin{lem}{lemmaxinorb}
For any $\gamma\in\fh^*$ one has
$$\begin{array}{l}
\gamma(\alpha^{\vee})\in\mathbb{Z}_{\geq 0}\text{  for each }\alpha\in\Pi_{pr}
\ \Longrightarrow \forall y\in W\ \ y\gamma\leq \gamma;\\
\gamma(\alpha^{\vee})\in\mathbb{Z}_{>0}\text{  for each }\alpha\in\Pi_{pr}
\ \Longrightarrow \forall y\in W\setminus\{Id\}\  \ y\gamma<\gamma.
\end{array}
$$
\end{lem}
\begin{proof}Take $\gamma$ such that $\gamma(\alpha^{\vee})\in\mathbb{Z}_{\geq 0}$
 for all $\alpha\in\Pi_{pr}$ and $y\in W\setminus\{Id\}$.
 Recall that $W$ is a Coxeter group. We will prove that  $y\gamma\leq \gamma$ by induction on the length $l(y)$. Let $y=wr_{\alpha}$ for some $\alpha\in\Pi_{pr}$ and
 $l(w)<l(y)$. Then $w\alpha\in \Delta^+$ (see~\cite{Kbook2}, Lem. 3.11).
 By induction assumption $w\gamma\leq\gamma$ and
 $$wr_{\alpha}\gamma=w\gamma-\gamma(\alpha^\vee)w\alpha\leq w\gamma\leq\gamma$$
 since $w\alpha\in \Delta^+$ and $\gamma(\alpha^\vee)\in\mathbb{Z}_{\geq 0}$. Moreover, if $\gamma(\alpha^\vee)\in\mathbb{Z}_{>0}$, then $y\gamma<\gamma$ and
 this implies the second assertion.
\end{proof}

\subsection{Types of Cartan matrices}\label{types}
Let $A$ be an indecomposable Cartan supermatrix.

\begin{defn}{}
We say that $A$ has  {\em non-isotropic type} if
 $a_{ii}\not=0$ for each $i$ and {\em isotropic type} otherwise.
\end{defn}

Clearly, this definition does not depend on a representative of the equivalence class of
supermatrices. In fact, the equivalence class of $(A,p)$ coincides with
$\{(A,p)\}$ if and only if $A$ has non-isotropic type  or $A=(0)$. In the latter case $\fg(A)=\fgl(1|1)$).

\subsubsection{}
A square matrix $B$ is called {\em symmetrizable}
if $DB$ is symmetric for some invertible diagonal matrix $D$.
Let $(A,p)$ be a Cartan supermatrix.
The matrix $A$ is symmetrizable if and only if
$\fg(A)$ admits a non-degenerate invariant bilinear form (in particular,
equivalent matrices are both symmetrizable or not). Such form
induces a  non-degenerate  $W$-invariant bilinear form $(-,-)$ on $\fh^*$.
For $\gamma\in\ol{\Delta}_{re}$ one has $(\gamma|\gamma)=0$
if and only if $\gamma$
is isotropic; if $\gamma$ is not isotropic, then
$\gamma^{\vee}=2\gamma/(\gamma|\gamma)$.

\subsubsection{}\label{FINAFFIND}
Let $C$ be an indecomposable real $n\times n$ matrix satisfying (A00)
and
$c_{ij}\in\mathbb{Z}_{\leq 0}$
for $i\not=j$. We view $C$ as a linear operator on $\mathbb{R}^n$;
for $v\in\mathbb{R}^n$ we write $v>0$ (resp., $v\geq 0$) if all coordinates of $v$
are positive (resp., non-negative). By Thm.4.3 of~\cite{Kbook2},
the matrix $C$ satisfies exactly one of the following conditions:

(FIN) $\det C\not=0$, there exists $u>0$ such that $Cu>0$ and $Cv\geq 0$ implies
$v>0$ or $v=0$;

(AFF) $\corank C=1$, there exists $u>0$ such that $Cu=0$ and $Cv\geq 0$ implies
$Cv=0$;

(IND) there exists $u>0$ such that $Cu<0$ and $Cv\geq 0, v\geq 0$ implies
$v=0$.

We say that $C$ has a type (FIN), (AFF) or (IND) respectively.

If $C$ is symmetrizable, then the diagonal matrix $D$
can be chosen in such a way that $d_{ii}>0$. For this choice
$C'=DC$  satisfies the above assumptions (on $C$) and $C, C'$
are of the same type.

\subsection{Non-isotropic type}\label{non-isotropictype}
Let $A$ have  non-isotropic type. Then $\cB=\{\Sigma\}$.  In this case $\fg_{pr}=\fg(B)$.
Furthermore, $b_{ij}=s^{-1}_ia_{ij}s_j$ where $s_i=1+p(i)$.  Therefore $A,B$ have the same type.

Consider a Kac--Moody Lie algebra $\fg^{ev}=\fg(A)$
   given by the same Cartan matrix $A$ and $p=0$.
 Denote by  $\Sigma^{ev}$ (resp., $\Delta^{ev}$)  the  base (resp.,
 the root system) of $\fg^{ev}$. Note that the roots systems of $\fg(B)$ and $\fg^{ev}$ are different as one can see from the following example.

\subsubsection{}
For example, for $\fg=\mathfrak{osp}(2|4)^{(2)}$
the Dynkin diagrams are

$$\begin{array}{ccc}
\Sigma & \Sigma^{ev}&\Pi_{pr}\\
\ &\ &\ \\
\overset{\delta-\vareps_{1}}{\bullet}\Longleftarrow
\overset{\vareps_{1}-\vareps_2}{\circ}
\Longrightarrow\overset{\vareps_2}{\bullet}\ \ &\ \
\overset{\delta-\vareps_{1}}{\circ}\Longleftarrow
\overset{\vareps_{1}-\vareps_2}{\circ}
\Longrightarrow\overset{\vareps_{2}}{\circ}\ \ &\ \
\overset{2\delta-2\vareps_{1}}{\circ}\Longrightarrow
\overset{\vareps_{1}-\vareps_2}{\circ}
\Longleftarrow\overset{2\vareps_2}{\circ}
  \end{array}
$$
so $\Sigma^{ev}$ is of type $D_3^{(2)}$ and $\Pi_{pr}$ is of type  $C_2^{(1)}$.

\subsubsection{}\label{symnoniso}
If we identify $ \fh^*$ with $(\fh^{ev})^*$, then
the Weyl groups of $\Delta$ and of $\Delta^{ev}$ coincide,
and we have  $\ol{\Delta}_{re}=\Delta^{ev}_{re}$.
Using  methods of~\cite{Kbook2}, Ch. V one can show that
$\Delta_{im}=\Delta_{im}^{ev}$.

If $A$ is symmetrizable, we  normalize the bilinear form $(-|-)$
in such a way that $(\alpha|\alpha)=2$ for some $\alpha\in\Sigma$; then
$(\alpha|\alpha)\in \mathbb{Q}_{>0}$ for every $\alpha\in\Sigma$.
From~\cite{Kbook2},  Prop. 5.2, it follows that $\alpha\in \Delta$ is imaginary if and only if
$||\alpha||^2\leq 0$.

\subsection{Isotropic type}\label{isotropictype}
Let $(A,p)$ be a Cartan supermatrix of isotropic type.
These  supermatrices
were classified in~\cite{H}; we recall the results
of this classification below.

The matrix $B$ (introduced in~\S~\ref{fgB}) has the following property:
{\it all  indecomposable components of $B$ are of the same type ((FIN), (AFF) or (IND))}. We have the
 following three cases according to the type of $B$:

(FIN)  All indecomposable components of $B$ are of type (FIN).
In this case $\dim\fg(A)<\infty$,
$\fg_{\ol{0}}=\fg_{pr}$ is reductive
and $[\fg_{\ol{0}},\fg_{\ol{0}}]=\fg(B)$. The group $W$
is finite.

(AFF)  All indecomposable components of $B$ are of type (AFF).
In this case $\fg(A)$ is of finite growth and $W$
is a direct product of affine Weyl groups.
If $A$ is symmetrizable, then
$\fg(A)$ can be described using (twisted) affinization of a
finite dimensional Kac--Moody superalgebra, \cite{vdL}.
If $A$ is not symmetrizable, then either $B=A_1^{(1)}$ and
$A$ is of type $S(1,2,a)$ for $a\not=0, a^{-1}\in\mathbb{C}/\mathbb{Z}$, or
$B=A_n^{(1)}$ and $A$ is of type $Q^{(2)}_{n}$.

(IND) All indecomposable components of $B$ are of type (IND). In this case $\fg(A)$ is of infinite growth. One has
$B=\begin{pmatrix} 2& m& m\\n& 2& n\\ p& p&2 \end{pmatrix}$, it is
indecomposable and symmetrizable
($m,n,p$ are negative integers, not all equal to $-1$);
any triple $(m,n,p)$ give rise to two supermatrices $A^+$ and $A^-$, which are not symmetrizable.

\subsubsection{}\label{delta'}
From this classification it follows that $B$ is always symmetrizable.
For (FIN) and (IND) the elements of $\Pi_{pr}$
are linearly independent and $\det B\not=0$.
For (AFF)  the elements of $\Pi_{pr}$
are linearly independent if and only if $B$ is indecomposable.
{For (FIN) all roots are real; for
(AFF) one has $\Delta^+_{im}=\mathbb{Z}_{>0}\delta$ (where
$\delta$ is the minimal imaginary root).}

\subsection{Examples}
If $A=(2)$ and $p(1)=\ol{0}$, then $\fg(A)=\fsl_2$;
if $A=(2)$ and $p(1)=\ol{1}$, then $\fg(A)=\fosp(1|2)$;
if $A=(0)$, then $p(1)=\ol{1}$ and $\fg(A)=\fgl(1|1)$.

Basic classical Lie superalgebras $\mathfrak{sl}(m|n)$ ($m\not=n$),
$\mathfrak{osp}(m|2n)$, $D(2,1,a)$, $F(4)$, $G(2)$ and their affinizations are
Kac--Moody. The Lie superalgebra
$$\mathfrak{gl}(m|m)=\mathfrak{psl}(m|m)+\mathbb{C}K+\mathbb{C}d$$
is Kac--Moody: its subquotient $\mathfrak{psl}(m|m)$ is simple (and
basic classical)
for $m\geq 2$, this is the only case in (FIN) when the Cartan matrix is degenerate
(and has corank $1$).

The corresponding affine Lie superalgebra $\mathfrak{gl}(m|m)^{(1)}$
$$\mathfrak{gl}(m|m)^{(1)}=\sum_{i\in\mathbb{Z}} \fgl(m|m)\otimes t^i
+\mathbb{C}K+\mathbb{C}d$$
is not a Kac--Moody superalgebra. The Kac--Moody superalgebra
corresponding to $A(n|n)^{(1)}$ is a subquotient of $\mathfrak{gl}(m|m)^{(1)}$, it can be described by
$$\mathbb{C}K+\mathbb{C}d+\mathbb{C}K'+\mathbb{C}d'+\mathfrak{psl}(m|m)\otimes\mathbb{C}[t,t^{-1}].$$
Double central extension is due to the fact that the Cartan matrix has corank $2$.

For the Cartan supermatrix ${Q}_{\ell}^{(2)}$ the algebra $\fg_{pr}=\fg_{\ol{0}}$ is isomorphic to $\fsl_{\ell-1}^{(1)}/(K)\times \mathbb{C}$
and $W$ is the affine Weyl group $A_{\ell-1}^{(1)}$.

The Cartan supermatrix $S(2,1,a)$ is a deformation of $A(1|0)^{(1)}$;
for these cases the algebra $\fg_{pr}\cong\fsl_{2}^{(1)}\times\mathbb{C}$ is a proper subalgebra of $\fg_{\ol{0}}$ and
$W$ is the affine Weyl group $A_1^{(1)}$.

For the remaining case $Q^{\pm}(m,n,t)$ the Cartan matrix $B$ is a symmetrizable
$3\times 3$ matrix, see~\cite{H}, Lem. 3.5 and $\fg_{pr}\cong\fg(B)$.

\subsection{Notations}\label{notat1}
Fix the base $\Sigma$.
We denote by $\Delta^+$ the set of positive roots and set
$$Q^+:=\mathbb{Z}_{\geq 0}\Sigma.$$

\subsubsection{Weyl vector}\label{useful}
We fix a Weyl vector $\rho\in\fh^*$  such that for each $\gamma\in\Sigma$ one has
$$\rho(\gamma^{\vee})=\gamma(\gamma^\vee).$$
If $\beta\in\Sigma$ is isotropic and
$\Sigma'=r_{\beta}\Sigma$ we can choose the  Weyl vector for
$\Sigma'$ as $\rho':=\rho+\beta$. We always choose $\rho_{\Sigma'}$ for all $\Sigma'\in\cB$
  following this rule.
This implies the following useful fact:
$\rho(\alpha^{\vee})$ is integer for {non-isotropic} $\alpha\in\Delta_{re}^+$
and this integer is odd if $\alpha$ is odd.

We define the twisted action of $W$ by
setting for $w\in W,\lambda\in\fh^*$:
$$w.\lambda=w(\lambda+\rho)-\rho.$$
(Note that the definition does not depend on the choice of $\rho$ if
$\Sigma$ is fixed).

\subsubsection{Weyl denominator}\label{Weyldenom}
We set
$$R_{\ol{0}}:=\prod_{\gamma\in \Delta^+_{\ol{0}}} (1-e^{-\gamma})^{\dim \fg_{\gamma}},\
 \ R_{\ol{1}}:=\prod_{\gamma\in \Delta^+_{\ol{1}}} (1+e^{-\gamma})^{\dim \fg_{\gamma}}.$$
Using geometric series we expand $R_{\ol{1}}^{-1}=\sum_{\nu\in Q^+} x_{\nu}e^{-\nu}$ (where $x_{\nu}\in\mathbb{Z}$)
and introduce the Weyl denominator by the formula
 $$R:=R_{\ol{0}}R_{\ol{1}}^{-1}=\sum_{\nu\in Q^+} y_{\nu}e^{-\nu}.$$

For a formal sum $\sum_{\nu\in\fh^*} a_{\nu} e^{\nu}$ with $a_{\nu}\in\mathbb{Q}$ we set
$$\begin{array}{l}
\supp (\sum_{\nu\in\fh^*} a_{\nu}e^{\nu}):=\{\nu|\ a_{\nu}\not=0\},\\
w(\sum_{\nu\in\fh^*} a_{\nu} e^{\nu}):=
\sum_{\nu\in\fh^*} a_{\nu} e^{w\nu}.\end{array}$$

\subsubsection{}
\begin{lem}{lemR}
If $\alpha\in \Pi_{pr}$ is such that $\alpha\in\Sigma$ or
$\frac{\alpha}{2}\in\Sigma$, then  $r_{\alpha}(Re^{\rho})=-Re^{\rho}$.
\end{lem}
\begin{proof}
{Set $\beta:=\alpha$
if $\alpha\in\Sigma$ and $\beta:=\frac{\alpha}{2}$ otherwise. Then
$\beta\in\Sigma$ and $R=(1-e^{-\beta})R_{\alpha}$, where }
$$R_{\alpha}=\prod_{\gamma\in \Delta^+_{\ol{0}}\setminus\{2\alpha\}} (1-e^{-\gamma})^{\dim \fg_{\gamma}}
\prod_{\gamma\in \Delta^+_{\ol{1}}\setminus\{\alpha\}} (1+e^{-\gamma})^{-\dim \fg_{\gamma}}.$$
The set $\Delta^+\setminus\mathbb{Z}\alpha$ is $r_{\alpha}$-invariant so $r_{\alpha}R_{\alpha}=R_{\alpha}$. {Now the formula $r_{\alpha}\rho=\rho-\beta$
implies the claim.}
\end{proof}

\subsection{Root subsystems }\label{Deltalambda}
A subset $\Delta'\subset W\Pi_{pr}$
is called a {\em root subsystem}
if $r_{\alpha}\beta\in\Delta'$ for any $\alpha,\beta\in\Delta'$.
For a root subsystem $\Delta'$ we denote by $W(\Delta')$ the subgroup
of $W$ generated by $r_{\alpha},\alpha\in\Delta'$; we set
$(\Delta')^+:=\Delta'\cap\Delta^+$ (recall that $\Delta^+\cap
W\Pi_{pr}$ does not depend on a choice of
$\Sigma'\in\cB$)
and introduce the base {$\Pi(\Delta')$ }by the formula
$$\Pi(\Delta'):=\{\beta\in (\Delta')^+|\ \not\exists\alpha\in \Delta'
\text{ s.t. }  0<r_{\alpha}\beta<\beta\},$$
where we use the partial order introduced in~\S~\ref{parord}.
In particular,
$(\Pi_{pr}\cap\Delta')\subset \Pi(\Delta')$. {By~\Lem{lemPilambda}
$\Pi(\Delta')$ coincides with the set of indecomposable elements in $(\Delta')^+$.}

For $\lambda\in\fh^*$ we introduce a root subsystem
$$\Delta(\lambda)=\{\alpha\in W\Pi_{pr}|\
\lambda(\alpha^{\vee})\in \mathbb{Z}\}.$$
and set  $W(\lambda):=W(\Delta(\lambda))$, $\Pi(\lambda):=\Pi(\Delta(\lambda))$.
For example, $\Delta(0)=W\Pi_{pr}$ and $\Pi(0)=\Pi_{pr}$.
One has $\Delta(w\lambda)=w\Delta(\lambda)$ for every $w\in W$
and $\Delta(\lambda)=\Delta(\lambda-\mu)$ for $\mu\in
\mathbb{Z}\Sigma$.

We have the following useful formula
$$\Delta(\lambda):=\{\alpha\in W\Pi_{pr}|\ \alpha/2\not\in\Delta\ \&\
(\lambda+\rho)(\alpha^{\vee})\in \mathbb{Z}\text{ or }
\ \alpha/2\in\Delta\ \&\
(\lambda+\rho)(\alpha^{\vee})\in\mathbb{Z}+1/2\}.$$
For every $w\in W(\lambda)$ we have $w(\lambda+\rho)-(\lambda+\rho)\in\mathbb{Z}\Sigma$.

\subsubsection{Example}\label{exB412}
Consider $\fg:=\mathfrak{osp}(9|2)$.
In this case
$$\Pi=\{\vareps_1-\vareps_2,\vareps_2-\vareps_3,\vareps_3-\vareps_4,\vareps_4;\delta_1\},\ \
\Pi_{pr}=\{\vareps_1-\vareps_2,\vareps_2-\vareps_3,\vareps_3-\vareps_4,\vareps_4;2\delta_1\},$$
so $\Pi_{pr}$ is of the type $B_4\times A_1$.
For $\lambda=\frac{1}{3}(\vareps_1+\vareps_3)$ one has
$$\Delta(\lambda)^+:=\{\vareps_1-\vareps_3,\vareps_2,\vareps_4,\vareps_2\pm\vareps_4,
2\delta_1\},$$
so $\Delta(\lambda)$ is  a root system  the type $A_1\times B_2\times A_1$ with the base
$$\Pi(\lambda)=\{\vareps_1-\vareps_3,\vareps_2-\vareps_4,\vareps_4;
2\delta_1\}.$$

\subsubsection{}
The following lemma is standard.

\begin{lem}{lemPilambda}
Let $\Delta'\subset W\Pi_{pr}$ be a root subsystem and
$\Pi':=\Pi(\Delta')$.

(i) The group $W(\Delta')$ is generated by $r_{\alpha}$ with
$\alpha\in  \Pi'$ and $\Delta'=W(\Delta') \Pi'$.

(ii) One has  $(\Delta')^+\subset \mathbb{Z}_{\geq 0}\Pi'$.

(iii) If $\alpha\in \Pi_{pr}\setminus\Pi'$, then
 $\Pi(r_{\alpha}\Delta')=r_{\alpha}\Pi'$.
\end{lem}
\begin{proof}
We will use the following observation:
if  $\gamma,\alpha\in (\Delta')^+$ are such that
$0<r_{\alpha}\gamma<\gamma$, then $\alpha<\gamma$
(since $r_{\alpha}\gamma=\gamma-j\alpha$ for some $j>0$).

For (i) let $W''$ be the subgroup of  $W$
which is generated by the $r_{\alpha}$s with
$\alpha\in   \Pi'$. Assume that $(\Delta')^+\not\subset W'' \Pi'$.
Let $\gamma$ be a minimal
 element in $(\Delta')^+\setminus W'' \Pi'$.
Since $\gamma\not\in \Pi'$ there exists $\alpha\in (\Delta')^+$
such that $0<r_{\alpha}\gamma<\gamma$. Then, by above,
 $r_{\alpha}\gamma,\alpha<\gamma$, so
$r_{\alpha}\gamma,\alpha\in W'' \Pi'$. In particular, $r_{\alpha}\in W''$ and thus
$\gamma\in  W'' \Pi'$, a contradiction. We conclude that
$(\Delta')^+\subset W'' \Pi'$.
Then  $r_{\alpha}\in W''$
 for every $\alpha\in\Delta'$, so $W''=W(\Delta')$. This establishes
 (i).  For (ii) let  $\gamma$ be a minimal
 element in $(\Delta')^+\setminus (\mathbb{Z}_{\geq 0} \Pi')$.
Since $\gamma\not\in \Pi'$, there exists $\alpha\in (\Delta')^+$
such that $0<r_{\alpha}\gamma<\gamma$.  Then
$r_{\alpha}\gamma,\alpha\in \mathbb{Z}_{\geq 0}\Pi'$,
so $\gamma\in \mathbb{Z}_{\geq 0}\Pi'$,
 a contradiction.
This gives (ii).

For (iii) take $\alpha\in \Pi_{pr}\setminus  \Pi'$ and
$\beta\in  \Pi'$. Assume that  $r_{\alpha}\beta\not\in  \Pi(r_{\alpha}\Delta')$. Since $r_{\alpha}\beta\in \Delta^+$ there exists
\begin{equation}\label{eqralph}
\alpha'\in \bigl(r_{\alpha}\Delta'\cap\Delta^+\bigr)\ \text{ s.t. }\
r_{\alpha'}r_{\alpha}\beta\in\Delta^+\text{ and }
(r_{\alpha}\beta)((\alpha')^{\vee})>0.\end{equation}
Set $\gamma:=r_{\alpha}\alpha'$. Then $\gamma\in (\Delta'\cap r_{\alpha}\Delta^+)$
and
$r_{\gamma}\beta=r_{\alpha}r_{\alpha'}r_{\alpha}\beta\in (\Delta'\cap r_{\alpha}\Delta^+)$.

By above, $\alpha\not\in\Delta'$, so $\Delta'\cap r_{\alpha}\Delta^+=(\Delta')^+$.
Thus $\gamma,r_{\gamma}\beta \in  (\Delta')^+$.
By~(\ref{eqralph}),
$$\beta(\gamma^{\vee})=\beta((r_{\alpha}\alpha')^{\vee})=(r_{\alpha}\beta)((\alpha')^{\vee})>0.$$
 Hence $\beta\not\in \Pi'$, a contradiction.
This completes the proof.
\end{proof}

\subsubsection{}\label{friends}
\begin{defn}{}
Let $\Delta'$ be  a root subsystem of $W\Pi_{pr}$.

For $\alpha_1,\ldots,\alpha_s\in \Pi_{pr}$ we say that
$\Delta'$
is $(\alpha_1,\ldots,\alpha_s)$-{\em friendly} if
$$\alpha_1,\ r_{\alpha_1}\alpha_2,\,\ldots,\
r_{\alpha_1}\ldots r_{\alpha_{s-1}}\alpha_s
\not\in\Delta'.$$

We say that $\Delta'$ is $w$-{\em friendly} for $w\in W$ if $w=Id$ or if
$w$ can be written as
$w=r_{\alpha_s}\ldots r_{\alpha_{1}}$ such that
$\Delta'$
is $(\alpha_1,\ldots,\alpha_s)$-{friendly}.
We say that $\lambda$ is $(\alpha_1,\ldots,\alpha_s)$-{\em friendly}
(resp., $w$-friendly) if $\Delta(\lambda)$ is
$(\alpha_1,\ldots,\alpha_s)$-{friendly}
(resp., $w$-friendly).
\end{defn}

\subsubsection{}
\begin{lem}{lemwmin}
Fix $\lambda\in\fh^*$.

(i) If $\alpha\in\Pi_{pr}, y\in W$ are such that
$\lambda$ is $r_{\alpha}$-friendly
and $r_{\alpha}\lambda$ is $y$-friendly, then
$\lambda$ is $yr_{\alpha}$-friendly.

(ii) If $\lambda$ is $w$-friendly, then $\Pi(w\lambda)=w\Pi(\lambda)$.

(iii) For each $\beta\in \Pi(\lambda)$
 there exists $w\in W$ such that $\lambda$ is $w$-friendly
and $w\beta\in\Pi_{pr}$.
\end{lem}
\begin{proof}
Note that for any $z\in W$ we have  $\Delta(z\lambda)=z\Delta(\lambda)$. This  implies (i);
 (ii) follows from~\Lem{lemPilambda}  (iii).
We prove (iii) by induction on $\beta\in(\Delta^+\cap W\Pi_{pr})$.
If $\beta\in\Pi_{pr}$ one has $w=Id$. Now take $\beta\in \Pi(\lambda)$ with
$\beta\not\in\Pi_{pr}$. Since $\beta\in W\Pi_{pr}$, the orbit
$W\beta$ contains $\beta'\in\Pi_{pr}$, so, by~\Lem{lemmaxinorb},
 there exists $\alpha\in\Pi_{pr}$ such that
$0<r_{\alpha}\beta<\beta$.  Since $\beta\in \Pi(\lambda)$
one has  $\alpha\not\in \Delta(\lambda)$.
Thus $\lambda$ is $r_{\alpha}$-friendly and, by (ii),
$r_{\alpha}\beta\in \Pi(r_{\alpha}\lambda)$.
Since $r_{\alpha}\beta<\beta$, the induction hypothesis
provides the existence of $y\in W$ such that $yr_{\alpha}\beta\in\Pi_{pr}$
and $y$ is $r_{\alpha}\lambda$-friendly.
By (i), $\lambda$ is $yr_{\alpha}$-friendly. This completes the proof.
\end{proof}

\subsubsection{}
Recall that $B$ is the "Cartan matrix of $\Pi_{pr}$''
see~\ref{fgB} and
that $B$ is symmetrizable if  the Cartan matrix $A$ is  of isotropic type
(see~\S~\ref{isotropictype}) or if $A$ is symmetrzaible.

\subsubsection{}
\begin{prop}{eq>0}
Assume that the Cartan matrix $A$ is either symmetrizable
or of isotropic type. Let $\Delta'\subset W\Pi_{pr}$ be a root subsystem . Then

(i) $\Pi(\Delta')=\{\beta\in (\Delta')^+|\ r_{\beta}((\Delta')^+\setminus\{\beta\})=
(\Delta')^+\setminus\{\beta\}\}$;

(ii) ${W}(\Delta')$ is the Coxeter group
corresponding to $\Pi(\Delta')$;

(iii) if $\Delta'=\Delta(\lambda)$, then the following conditions are equivalent:

\hspace{1cm}  $(\lambda+\rho)(\alpha^{\vee})>0$ for each $\alpha\in\Pi(\lambda)$;

\hspace{1cm} $\lambda+\rho$ is a maximal element in ${W}(\lambda)(\lambda+\rho)$
and $\Stab_{W(\lambda)}(\lambda+\rho)=\{Id\}$.
\end{prop}
\begin{proof}
By~\cite{KT98}, 2.2.8--2.2.9, if $B$ is symmetrizable, then (i) and (ii) holds.
{Using (ii) and~\Lem{lemmaxinorb} we
obtain (iii).}
\end{proof}

\subsection{Characters and supercharacters}\label{Nnu}
Let $N$ be a $\fg$-module which is locally finite over $\fh$. We set
$$N_{\nu}:=\{v\in N|\ \forall h\in\fh\ \exists s\ (h-\nu(h))^sv=0\},\ \
\supp(N):=\{\nu\in\fh^*|\ N_{\nu}\not=0\}$$
and say that $v$ has ``weight $\nu$'' if $v\in N_{\nu}$.
Take $N$ such that  $\mu-\nu\in\mathbb{Z}\Delta$
for all $\mu,\nu\in\supp(N)$ (this holds
for every  indecomposable module). Then $\Delta(\nu)=\Delta(\mu)$
for all $\nu,\mu\in\supp(N)$. We set
$$\Delta(N):=\Delta(\nu)=
\{\alpha\in W\Pi_{pr}|\ \forall\nu\in \supp(N)\ \
\nu(\alpha^{\vee})\in \mathbb{Z}\},\ \ \Pi(N):=\Pi(\lambda),\ \ W(N):=W(\lambda)$$
($\Delta(N)$ is a root subsystem of
$W\Pi_{pr}$).
Note that $\Delta(N)=\Delta(\Res^{\fg}_{\fg_{pr}}N)$.

For a supervector space $E$ we write
$$\sdim E=\dim E_{\ol{0}}+\epsilon \dim E_{\ol{1}},$$
where $\epsilon$ is a formal variable satisfying $\epsilon^2=1$.
If all weight spaces
$N_{\nu}$ are finite-dimensional, we introduce
$$\ch_{\epsilon} N:=\sum_{\nu\in\fh^*} \sdim N_{\nu}e^{\nu}.$$

The usual character (resp., supercharacter) is obtained from $\ch_{\epsilon} N$ by evaluating
$\epsilon=1$ (resp., $\epsilon=-1$):
$$\ch N=\sum_{\nu\in\fh^*} \dim N_{\nu}e^{\nu}$$
and $\supp(N)=\supp(\ch N)$.

For every indecomposable module $N$ over Kac-Moody superalgebra
$\ch N$ and $\sch N$ are related as follows.
Take $\lambda\in \supp(N)$. Then
$$\supp(e^{-\lambda}\ch N)\subset \mathbb{Z}\Sigma$$
and
$$e^{-\lambda}\ch N=\pm \pi(e^{-\lambda}\sch N),$$
where the automorphism $\pi$ is defined by
$\pi(e^{\mu}):=p(\mu) e^{\mu}$ for $\mu\in\mathbb{Z}\Sigma$.
This follows form the fact that every weight space of $N$ is either
even or odd.

\subsection{Categories $\CO^{fin},\CO$ and $\CO^{inf}$}\label{COK}
Fix $\Sigma\in\cB$ and set $\fn^+:=\oplus_{\alpha\in\Delta^+}\fg_{\alpha}$.

We denote by $\CO$ the category $O$ introduced in~\cite{DGK}:
this is a full category of $\fg$-modules $N$ satisfying
the following conditions: $\fh$ acts diagonally with  finite-dimensional
eigenspaces  and
$\supp N\subset \bigcup_{i=1}^ s (\lambda_i-Q^+)$, where
$\{\lambda_i\}_{i=1}^s\subset\fh^*$.
We denote by $\CO^{fin}$ the full subcategory of $\CO$
consisting of finitely generated modules (this category is  introduced in~\cite{BGG}).
The category $\CO$ is  equipped with an {involutive}  contragredient
duality functor $N\mapsto N^{\#}$.
If $\fg$ is finite-dimensional, this functor  restricts to a duality functor on $\CO^{fin}$.

\subsubsection{}\label{hwt}
\begin{defn}{defnOinf}
Let ${\CO}^{inf}$ be the full category of $\fg$-modules $N$ with the
following properties:

(C1) $\fh$ acts diagonally on $N$;

(C2) every $v\in N$ generates a finite-dimensional $\fn^+$-submodule of $N$.

\end{defn}

For a $\fg$-module $N$
we call a vector $v\in N$ {\em singular} if $v$ is an
$\fh$-eigenvector and
$\fn^+v=0$. We say that $N$ is a highest weight module if
$N$ is generated by a singular vector.

\subsubsection{}\label{Ofininf}
One has $\CO^{fin}\subset \CO\subset \CO^{inf}$.

One readily sees that $N\in {\CO}^{inf}$ if and only if
any cyclic submodule of $N$ lies in the category $\CO^{fin}$; thus
${\CO}^{inf}$ is the inductive completion of $\CO^{fin}$.

On the other hand, $N\in\CO^{fin}$ if and only if $\fh$ acts diagonally
on $N$ and $N$ admits a finite filtration whose quotients are highest weight modules,
see~\Lem{lemtildeCO}.

It is easy to see that highest weight modules lie in $\CO^{fin}$
for any choice of $\Sigma'\in\cB$, see~\cite{S3}, Lemma 10.2. As a result,
the categories $\CO^{fin}$ and $\CO^{inf}$ do not depend
on the choice of $\Sigma'\in\cB$. By contrast, the category
$\CO$ depends on the choice of $\Sigma'\in\cB$.

\subsubsection{}
\begin{lem}{lemOOpr}
  Let $\fg$ have non-isotropic type.
For  $N\in\CO(\fg)$ we have
$\Res^{\fg}_{\fg_{pr}}N\in\CO(\fg_{pr})$.
\end{lem}
\begin{proof}
  Set $Q^+_{pr}:=\mathbb{Z}_{\geq 0}\Pi_{pr}$. {If $\fg$ has non-isotropic type}, then the quotient $Q^+/Q^+_{pr}$  consists of the sums
  $\sum_{\alpha\in \Sigma} x_{\alpha}\alpha$, where $x_{\alpha}\in
  \{0,1\}$ and $x_{\alpha}=0$ for even $\alpha$;
  in particular, this quotient is finite and thus
  $\Res^{\fg}_{\fg_{pr}}N\in\CO(\fg_{pr})$.
  \end{proof}

\subsubsection{}\label{ODelta}
For $\lambda\in\fh^*$ we denote by $M(\lambda)$
the Verma module of  highest weight $\lambda$ and
by $L(\lambda)$ its unique simple quotient. Every
simple module in $\CO$  is isomorphic to
$L(\lambda)$ for some $\lambda\in\fh^*$;
these modules  are self-dual with respect to the contragredient
duality functor  ($L(\lambda)\cong L(\lambda)^{\#}$).
By~\cite{DGK}, Prop. 3.4, the notion $[N:L(\lambda)]$
is well-defined for each object $N\in\CO$ and
$$\ch N=\sum_{\lambda\in\fh^*} [N:L(\lambda)]\ch L(\lambda).$$

Retain notation of~\S~\ref{Weyldenom}.
One has $R\ch M(\lambda)=e^{\lambda}$, $\ch L(\lambda)=\sum_{\mu\leq \lambda}b_\mu\ch M(\mu)$ with $b_\mu\in\mathbb Z$ and hence
$$\ch N=\sum_{\lambda\in\fh^*} a_{\lambda}\ch M(\lambda),
\ \ \ R\ch N=\sum_{\lambda\in\fh^*} a_{\lambda}e^{\lambda}
$$
for some coefficients $a_{\lambda}\in\mathbb{Z}$. Note that the first
identity
is well-defined due to restriction on $\supp(N)$.

Retain notations of~\S~\ref{Deltalambda}. For a
 root subsystem  $\Delta'\subset W\Pi_{pr}$  we denote by
$\CO_{\Delta'}$ (resp.,  of $\CO_{\Delta'}^{fin}$)
the full subcategory of $\CO$ (resp., $\CO^{fin}$)
with the indecomposable objects $N$ such that
$\Delta(N)=\Delta'$.
Since for any indecomposable module $N$
one has $\Delta(N)=\Delta(\lambda)$ if $[N:L(\lambda)]\not=0$,
the category  $\CO$ (resp., $\CO^{fin}$) is the direct sum of
the subcategories $\CO_{\Delta'}$ (resp.,  $\CO_{\Delta'}^{fin}$).

\subsubsection{Typical weights}\label{lambdalambda'}
For every $\Sigma'\in\cB$ we denote by $M(\lambda;\Sigma')$ and by
$L(\lambda;\Sigma')$  the Verma module and the simple module
for the base $\Sigma'$.
Let $\Sigma'=r_{\beta}\Sigma$ (where $\beta\in\Sigma$ is isotropic)
and $\rho,\rho'$ be the corresponding
Weyl vectors ($\rho'=\rho+\beta$).
If $(\lambda+\rho)(\beta^{\vee})\not=0$, then
$M(\lambda;\Sigma)=M(\lambda';\Sigma')$ and $L(\lambda;\Sigma)=L(\lambda';\Sigma')$ where
 $\lambda+\rho=\lambda'+\rho'$; if $(\lambda+\rho)(\beta^{\vee})=0$, then
$L(\lambda;\Sigma)=L(\lambda;\Sigma')$.

We call a weight $\lambda\in\fh^*$   {\em typical } if
$\lambda(\alpha^{\vee})\not=0$ for every isotropic root $\alpha\in\Delta_{re}$.
Note that the set of typical weights is $W$-invariant.
The simple highest weight module $L(\lambda)$ is called {\em typical }
if $\lambda+\rho$ is a typical weight.
By above, if $L(\lambda)=L(\lambda;\Sigma)$ is  typical,
then for any $\Sigma'\in\cB$ one has
$L(\lambda;\Sigma)=L(\lambda';\Sigma')$ where
 $\lambda+\rho=\lambda'+\rho'$; therefore the notion of typicality
 for a simple highest weight module does not depend on the choice of $\Sigma'\in\cB$.

\subsection{Symmetrizable Kac-Moody superalgebras}
\label{KM}
From now on till the end of this section
$\fg$ is a symmetrizable Kac-Moody superalgebra, so
 $\fg$ admits a non-degenerate invariant bilinear form which
induces a  non-degenerate  $W$-invariant bilinear form $(-|-)$ on $\fh^*$.

\subsubsection{}\label{delta}
We say that $\alpha\in\Delta$
is isotropic (resp., non-isotropic) if $(\alpha|\alpha)=0$
(resp., $(\alpha|\alpha)\not=0$); it is easy to see that this definition
agrees with the definition given above for  the real roots of
the general Kac--Moody algebra.
Using identification  $\fh^*\iso\fh$ given by $(-|-)$, we
  consider $\alpha^\vee$ as an element in $\fh^*$. Note that $\alpha^{\vee}$ is proportional to $\alpha$ for any real root $\alpha$.
  If $\alpha$ is non-isotropic real root we have
  $$\alpha^{\vee}:=\frac{2\alpha}{(\alpha|\alpha)}.$$

We  call $\lambda\in\fh^*$ {\em critical } if
$2(\lambda|\alpha)\in \mathbb{Z}||\alpha||^2$ for some imaginary root
$\alpha$ and {\em non-critical} otherwise. {The usual definition
differs by a shift by $\rho$}.)
The set of critical weights is ${W}$-invariant.

Recall that if $\fg$ is affine, then
$\Delta_{im}=\mathbb{Z}\delta\setminus\{0\}$, where $\delta$ is the minimal positive imaginary root; one has
$(\delta|\Delta)=0$, that is $(\Delta_{im}|\Delta)=0$.

\subsection{Blocks in $\CO$}\label{blocks}
Consider an equivalence relation on the set of isomorphism classes of simple modules in $\CO$ generated
by the set
$$\{(L_1,L_2)|\ L_1,L_2\text{ are simple in } \CO,\  \Ext^1_{\CO}(L_1,L_2)\not=0\}.$$
For each class of this equivalence relation the corresponding
block  is the full subcategory of $\CO$ with the objects $N$ such that
all simple subquotients of $N$ belong to this class.

\subsubsection{}\label{cK}
Set  $\ol{\Delta}^+:=\{\alpha\in\Delta^+|\ \frac{\alpha}{2}\not\in\Delta\}$ and

\begin{equation}\label{eqcK}
\begin{array}{ll}
\cK:=&\{(\lambda,\lambda-m\alpha)\in \fh^*\times\fh^*|\ \alpha\in\ol{\Delta}^+,\
m\in\mathbb{Z}_{>0}\ \text{ s.t. } 2(\lambda+\rho|\alpha)=m(\alpha|\alpha),\\
&\ \text{and $m$ is odd if $\alpha$
is odd and non-isotropic, $m=1$ if $\alpha$ is isotropic}
\}\end{array}\end{equation}

Note that for a non-critical typical $\lambda$ one has $(\lambda,\lambda')\in\cK$
if and only if $\lambda'\leq \lambda$ and
$\lambda'=r_{\alpha}\lambda$ for $\alpha\in \Delta(\lambda)$.

\subsubsection{}\label{equiv}
We denote by $\succeq$  the  reflexive and transitive relation on $\fh^*$ generated by $\cK$
and by $^{\sim}_{\CO}$ the equivalence relation  on $\fh^*$ generated by $\cK$
(this means that $\succeq$  is the minimal reflexive and transitive relation
such that $(\lambda,\nu)\in \cK$ implies
$\lambda\succeq\nu$ and $^{\sim}_{\CO}$ is the minimal equivalence relation  on $\fh^*$
with the similar property). 

From~\cite{KK}, Thm. 2 and Prop. 4.1 (the proofs are the same for superalgebras) we have

\begin{equation}\begin{array}{l}\label{KKthm}
[M(\lambda):L(\nu)]\not=0\ \Longrightarrow\ \lambda\succeq\nu;\\
(\lambda,\lambda-m\alpha)\in \cK\ \Longrightarrow\
\Hom_{\fg}(M(\lambda-m\alpha), M(\lambda))\not=0.\end{array}
\end{equation}
If $\fg$ has non-isotropic type, then
$[M(\lambda):L(\nu)]\not=0$ if and only if  $\lambda\succeq\nu$.

\subsubsection{}
\begin{cor}{corblock}
 The modules $L(\lambda), L(\nu)$ lie in the same block  if and only if
 $\lambda\,^\sim_{\CO}\,\nu$.
\end{cor}
\begin{proof}
Take  $(\lambda,\lambda-m\alpha)\in \cK$.
By~(\ref{KKthm}),
$[M(\lambda):L(\lambda-m\alpha))\not=0$. Since
$M(\lambda)$ is indecomposable, $L(\lambda), L(\lambda-m\alpha)$ are in the same block.
Thus $L(\lambda), L(\nu)$ lie in the same block  if
 $\lambda\,^\sim_{\CO}\,\nu$.

Now let $L(\lambda),L(\nu)$ be such that $\nu\not=\lambda$ and
$\Ext^1_{\CO}(L(\lambda),L(\nu))\not=0$. Let
$$0\to L(\nu)\to N\to L(\lambda)\to 0$$
be a non-splitting exact sequence.
If $\nu-\lambda\not\in Q^+$, then
$N$ is a quotient of $M(\lambda)$ and $\lambda\succeq\nu$ by~(\ref{KKthm}).
If $\nu-\lambda\in Q^+$ the same argument for $N^{\#}$ gives $\nu\succeq\lambda$.
Thus in both cases $\lambda\,^\sim_{\CO}\,\nu$.
Hence $L(\lambda),L(\nu)$ are equivalent in the sense of~\ref{blocks} implies $\lambda\,^\sim_{\CO}\,\nu$.
\end{proof}

\subsubsection{}
\begin{lem}{tipi}
If $\lambda\,^\sim_{\CO}\,\lambda'$, then $||\lambda||^2=||\lambda'||^2$ and

$\lambda$ is critical $\Longrightarrow\ $ $\lambda'$ is critical;

$\lambda$ is typical $\Longrightarrow\ $ $\lambda'$ is typical.
\end{lem}
\begin{proof} Let us prove the first statement. For affine case it is trivial
{since $(\Delta|\Delta_{im})=0$ }
by~\S~\ref{delta}{;} hence the statement is clear for isotropic type.
Now consider $\fg$ of non-isotropic type.
It suffices to check the case  $\lambda'=\lambda-m\alpha$,
where
 $2(\lambda|\alpha)=m(\alpha|\alpha)$ and $m\in\mathbb{Z}$
 such that $m$ is even if $\alpha$ is odd and non-isotropic.

 If  $\alpha$ is real, then
 $\lambda'=r_{\alpha}\lambda$. Since the  set of imaginary roots is $W$-invariant
$\lambda$ is critical implies that $\lambda'$
is critical. If $\alpha$ is imaginary, then
$2(\lambda'|\alpha)=-m(\alpha|\alpha)$, so
$\lambda,\lambda'$ are both critical.

The second assertion is similar.
 \end{proof}

\subsubsection{}\label{equiv}
Let $\Lambda$ be an equivalence
class with respect to $^\sim_{\CO}$.
We denote by $\CO_{\Lambda}$ the full subcategory
of $\CO$  whose objects $N$ satisfy the property:
$[N:L(\lambda)]\not=0$
implies $(\lambda-\rho)\in\Lambda$. 
By~\Cor{corblock}, $\CO_{\Lambda}$
is a block in the category $\CO$.
By Theorem 4.2,~\cite{DGK} (the statement and the proof are the same for superalgebras)
the category $\CO$ is the direct sum of categories $\CO_{\Lambda}$.

One has

\begin{equation}\label{suppN}
\begin{array}{ll}
\text{ for } M(\lambda)\in\CO_{\Lambda}\ & \supp(Re^{\rho}\ch M(\lambda))\subset \Lambda\cap (\lambda-Q^+),\\
\text{ for } N\in\CO_{\Lambda}\ &
\supp(Re^{\rho}\ch N)\subset \Lambda.\end{array}
\end{equation}

\subsubsection{}\label{modules}
The  module $L(\lambda)$ is called {\em non-critical }
if $\lambda+\rho$ is a non-critical weight. By~\S~\ref{isotropictype}
 any symmetrizable Kac--Moody superalgebra of isotropic type
 is either finite-dimensional or affine, so the criticality of a simple module $L\in\CO$
does not depend on the choice of  $\Sigma'\in\cB$ (if
 the module $L:=L(\lambda,\Sigma)$ is critical,
then it is critical for every $\Sigma'\in\cB$).

We call a module $N\in\CO$ {\em critical} (resp., {\em non-critical })
if $[N:L(\lambda)]\not=0$ implies that $L(\lambda)$ is critical (resp., non-critical).
We call a module $N\in\CO$ {\em typical} (resp., {\em atypical})
if $[N:L(\lambda)]\not=0$ implies that $L(\lambda)$ is typical (resp., atypical).
We denote by $\CO_{crit}$ (resp., $\CO_{noncrit}, \CO_{typ}, \CO_{atyp}$)
the full subcategory of $\CO$ consisting of critical (resp., non-critical,  non-critical
typical; non-critical atypical) modules.

\subsubsection{}
\begin{cor}{}
One has $\CO=\CO_{crit}\oplus \CO_{noncrit}$ and $\CO_{noncrit}=\CO_{typ}\oplus \CO_{atyp}$.
\end{cor}

\subsubsection{Remark}\label{typnoncrit}
If $\lambda$ is typical and non-critical, then
the equivalence class $\Lambda$ containing $\lambda$ (see~\S~\ref{cK}) is equal to
${W}(\lambda)\lambda$.

Let $X\subset\fh^*$ be the set of typical non-critical
 weights $\nu$ satisfying $\nu(\alpha^{\vee})\geq 0$
 for all  $\alpha\in\Pi(\nu)$. Then $\nu$ is maximal
 in $W(\nu)\nu$ and the argument of Thm. 5.5 of~\cite{AKMPP} shows that
if $N\in\CO$ is such that
any simple subquotient of $N$ is of the form $L(\lambda-\rho)$ for some
$\lambda\in X$, then $N$ is completely reducible.

\subsection{}
\begin{prop}{supptyp}
 Let $\fg$ be a symmetrizable Kac-Moody superalgebra.  For
   a typical simple highest weight module $L$
the expression $Re^{\rho}\ch L$ does not depend on $\Sigma\in\cB$.
\end{prop}
\begin{proof}
  The statement is not {tautological   }
  only when $\fg$ has   isotropic type.
  In this case, $\fg$ is finite-dimensional or affine.
  Take $\Sigma\in\cB$ and let $\Sigma':=r_{\beta}\Sigma$, where
  $\beta\in\Sigma$ is isotropic; we denote by $R'$ the Weyl
  denominator for $\Sigma'$. Recall that $\rho'=\rho+\beta$.
One has
$$(1+e^{-\beta})Re^{\rho}=(1+e^{-\beta})R'e^{\rho'}.$$
 By~(\ref{suppN})
$$Re^{\rho}\ch L=\sum_{\lambda\in \Lambda} a_{\lambda}
e^{\lambda},\ \ \
R'e^{\rho'}\ch L=
\sum_{\nu\in \Lambda} a'_{\nu}
e^{\nu}.
$$
and thus
$$\sum_{\lambda\in\Lambda} a_{\lambda}
(e^{\lambda}+e^{\lambda-\beta})=\sum_{\lambda\in\Lambda} a'_{\lambda}
  (e^{\lambda}+e^{\lambda-\beta}).$$

    Take $\lambda,\nu\in\Lambda$. By~\ref{equiv}
    $||\lambda||^2=||\nu||^2$.
    If $\lambda=\nu-\beta$, then
    $(\lambda|\beta)=0$, so $\lambda$ is atypical. By~\Lem{tipi}, this
   contradicts to typicality of $L$.
     \end{proof}

\section{Enright functor}
In this section we recall the construction and main properties
of Enright functor (see~\cite{En}, \cite{D}). We assume that
$\fg$ is a Lie superalgebra  containing a subalgebra $\fsl_2\subset\fg$
which acts locally finitely on $\fg$. We also fix a standard basis $\{e,h,f\}$ of $\fsl_2$.
By our assumptions $\ad e,\ad f$ are locally nilpotent and $\ad h$ is diagonalizable.

In this section we do not assume
that $\fg$ is Kac-Moody. Our approach simplifies some proofs for
$a\not\in\mathbb{Z}$.

We denote by $U(\fg)$ the universal enveloping algebra of $\fg$.

\subsection{Twisted localization functor}\label{Mathieufunctor}
Let us recall the twisted localization functor introduced by O.~Mathieu in~\cite{M}.

\subsubsection{}\label{dD}
 For $s\in\mathbb{Z}_{\geq 0}$ and $u\in U(\fg)$ one has
$$f^su=\sum_{i=0}^{\infty} \binom{s}{i} (ad\ f)^i(u) f^{s-i}.$$
Combining locally nilpotency of $\ad f$ on $U(\fg)$ and the above formula,
one easily sees that  $\{f^s\}_{s=0}^{\infty}$ form an Ore set in
$U(\fg)$; since $f\in\fg_{\ol{0}}$, $f$ is a non-zero divisor in $U(\fg)$.
We denote by $U(\fg)[f^{-1}]$ the localization of $U(\fg)$ by the Ore set
$\{f^s\}_{s=0}^{\infty}$.
Note that $U(\fg)\subset U(\fg)[f^{-1}]$ and that
$\ad f$  acts locally nilpotently on $U(\fg)[f^{-1}]$.

If $f=[F,F]$ for some $F\in\fg_{\ol{1}}$, then
$\{F^s\}_{s=0}^{\infty}$ is an Ore set (since, by above,
$\{F^{2s}\}_{s=0}^{\infty}$ is an Ore set) and the
localization of $U(\fg)$ by the Ore set
$\{F^s\}_{s=0}^{\infty}$ is equal to $U(\fg)[f^{-1}]$.

For a $\fg$-module $N$ we denote by $N[f^{-1}]$ the localized $U(\fg)[f^{-1}]$-module
$$N[f^{-1}]:=U(\fg)[f^{-1}]\otimes_{U(\fg)} N;$$
the functor $N\mapsto U(\fg)[f^{-1}]$ is a covariant exact additive functor.
The canonical map $N\to N[f^{-1}]$ ($v\mapsto 1\otimes v$) has the kernel
$$N^f:=\{v\in N| f^sv=0\ \text{ for }s>>0\}.$$

The algebra $U(\fg)[f^{-1}]$ admits inner automorphisms
$$\phi_s(u):=f^suf^{-s}=\sum_{i=0}^{\infty} \binom{s}{i} (ad f)^i(u) f^{-i}$$
for $s\in\mathbb{Z}$; this gives rise to a one-parameter group of automorphisms
\begin{equation}\label{phia}
\phi_a(u):=\sum_{i=0}^{\infty} \binom{a}{i} (ad f)^i(u) f^{-i}\end{equation}
(the sum is finite since $\ad f$  acts locally nilpotently on $U(\fg)[f^{-1}]$).
We have $\phi_a\circ \phi_b=\phi_{a+b}$.

Take $a\in\mathbb{C}$.
Let $U'f^a$ be the following $U(\fg)[f^{-1}]$-bimodule: as a left module
this is a free left module generated by a vector denoted by
 $f^a$ and the right action is given by
$$(uf^a)u':=u\phi_a(u')f^a$$
for $u,u'\in U(\fg)[f^{-1}]$. One readily sees that $U'f^a$ is a free right module generated by $f^a$. For $s\in\mathbb{Z}$ the map $uf^a\mapsto uf^s(f^{a-s})$ provides a bimodule isomorphism
$U'f^a\iso U'f^{a-s}$; we identify these bimodules and use the notation
$f^{a+s}=f^s(f^a)=f^a(f^s)\in U'f^a$.

For a $\fg$-module $N$ the functor
$$N\mapsto U'f^a\otimes_{U(\fg)[f^{-1}]}U(\fg)[f^{-1}]\otimes_{U(\fg)}N=U'f^a\otimes_{U(\fg)[f^{-1}]}N[f^{-1}].
$$
is a covariant exact additive functor from $U(\fg)-Mod$ to $U(\fg)[f^{-1}]-Mod$.
The twisted localization functor $\dD_a: \fg-Mod\to \fg-Mod$ is the composition of
this functor and the restriction functor  $U(\fg)[f^{-1}]-Mod\to U(\fg)-Mod$.
The map $v\mapsto f^a\otimes v$ gives a linear isomorphism $N[f^{-1}]\to \dD_a(N)$;
for $u \in U(\fg)[f^{-1}]$ one has
$uv\mapsto f^a\otimes uv=\phi_a(u)v$.

Another way to define $\dD_a$ is as a composition of localization and the twist by automorphism
$\phi_a$: $\dD_a(N)= (U(\fg)[f^{-1}]\otimes_{U(\fg)}N)^{\phi_a}$.

\subsubsection{}\label{propertydD}
The twisted localization $\dD_a$ is a covariant exact additive endofunctor in the category of $\fg$-modules.
The map $f^a\otimes f^b\mapsto f^{a+b}$ induces a bimodule isomorphism
$$U'f^{a}\otimes_{U(\fg)}U'f^{b}\to U'f^{a+b}$$
which gives
$$\dD_a\circ \dD_b\iso \dD_{a+b}.$$
In particular, $\dD_{a+s}(N)\cong \dD_a(N)$ for every $s\in\mathbb{Z}$.

Any vector in $\dD_a(N)$ is of the form $f^{a+j}\otimes v$ for some
$v\in N$ and $j\in\mathbb{Z}_{<0}$.
Let  $\fg'\subset\fg$ be a subalgebra which contains $\fsl_2$.
By PBW Theorem $U'f^{a}$ is a free $U(\fg')$-module. Hence the functor
$\dD_a$ commutes with the
restriction functor $\Res^{\fg}_{\fg'}$.

\subsubsection{}Let us compute the value of twisted localization functor on Verma modules over rank one  Lie superalgebras $\fsl_2$ and
$\mathfrak{osp}(1|2)$.  For $\mathfrak{osp}(1|2)$ one has
$\mathfrak{osp}(1|2)_{\ol{0}}=\fsl_2$ and  $\mathfrak{osp}(1|2)_{\ol{1}}$
is spanned by $E,F$ with the relations:
$$E^2=e,\ F^2=-f,\ [E,F]=h,\ [f,E]=F,\ [e,F]=e.$$

We consider the standard triangular decomposition: $\fh=\mathbb{C}h$ and
$\fn^+=\mathbb{C}e$ for $\fsl_2$, $\fn^+=\mathbb{C}e+\mathbb{C}E$
for $\mathfrak{osp}(1|2)$; we set
 $\fb:=\fh+\fn^+$, $\fb^-:=\fh+\fn^-$; for $b\in\mathbb{C}$
 we denote by $M_{\fb}(b), M_{\fb^-}(b)$ the corresponding Verma modules
(we always assume that the highest weight vectors are even).
In this section we denote by $\Pi$ the parity change functor.

We will use the following result.

\begin{lem}{lemtwVerma}
(i) Take $\fg=\fsl_2$. The  $h$-eigenspaces of $\dD_a(M_{\fb}(b))$
are one-dimensional
and the eigenvalues are $\{b-2a+2\mathbb{Z}\}$. The module
$\dD_a(M_{\fb}(b))$ is indecomposable.

The module $\dD_a(M_{\fb}(b))$ is simple if
$a, a-b\not\in\mathbb{Z}$.

If $a\in\mathbb{Z}$, then
$\dD_a(M_{\fb}(b))\cong M_{\fb}(b)[f^{-1}]$ has  the short exact sequence
$$0\to M_{\fb}(b)\to \dD_a(M_{\fb}(b))\to M_{\fb^-}^{\#}(b+2)\to 0.$$

If $a-b\in\mathbb{Z}$ and $a\not\in\mathbb{Z}$, then
$\dD_a(M_{\fb}(b))$ has a singular vector $f^{b+1}\otimes v$, where
$v\in M_{\fb}(b)$ is a highest weight vector and this gives a short exact sequence
$$0\to M_{\fb}(-b-2)\to \dD_a(M_{\fb}(b))\to M_{\fb^-}^{\#}(-b)\to 0.$$

(ii)  Take $\fg=\mathfrak{osp}(1|2)$. The  $h$-eigenspaces of $\dD_a(M_{\fb}(b))$
are one-dimensional
and the eigenvalues are $\{b-2a+\mathbb{Z}\}$. The module
$\dD_a(M_{\fb}(b))$ is indecomposable.

The module $\dD_a(M_{\fb}(b))$ is simple if
$a, a-b\not\in\mathbb{Z}$.

If $a\in\mathbb{Z}$, then
$\dD_a(M_{\fb}(b))\cong M_{\fb}(b)[f^{-1}]$ has the short exact sequence
$$0\to M_{\fb}(b)\to \dD_a(M_{\fb}(b))\to \Pi(M_{\fb^-}^{\#}(b+1))\to 0.$$

If $a-b\in\mathbb{Z}$ and $a\not\in\mathbb{Z}$, then
$\dD_a(M_{\fb}(b))$  has a singular vector $f^{b}F\otimes v$, where $v\in M_{\fb}(b)$ is a highest weight vector and this gives a short exact sequence
$$0\to \Pi(M_{\fb}(-b-1))\to \dD_a(M_{\fb}(b))\to M_{\fb^-}^{\#}(-b)\to 0.$$
\end{lem}
\begin{proof}
From~(\ref{phia}) we obtain the following formulae
in the bimodule $U'f^c$:
\begin{equation}\label{sl2formula}
hf^c-f^ch=-2c f^c,\  ef^c=f^ce+cf^{c-1}(h-(c-1))
\end{equation}
and
$$f^cF=Ff^c,\ f^cE=Ef^c+cFf^{c-1},\  f^cFE=(h+c)f^c-EFf^c$$
for the case $\mathfrak{osp}(1|2)$.
Take $v$ such that  $ev=0, hv=bv$. We obtain
$e(f^cv)=0$ if and only if $c=0$ or $c=b+1$; this proves (i).
For the case $\mathfrak{osp}(1|2)$
take $v$ such that  $Ev=0, hv=bv$.
For $c\not=0$ this gives $Ef^cv\not=0$ and
$E(Ff^c v)=0$ if and only if $c=b$; this proves (ii).
\end{proof}

\subsection{Enright functor}\label{En1}

Consider
the Zuckerman functor
$$\Gamma_e(N):=\{v\in N|\ \dim (\mathbb{C}[e]v)<\infty\}.$$
We introduce the Enright functor by $\cC_a:=\Gamma_e\circ \dD_a$, that is
$$\cC_a(N):=\{v\in \dD_a(N)|\ \dim (\mathbb{C}[e]v)<\infty\}.$$

For $b\in\mathbb{C}$ we denote by $\mathcal M(b)$ the full
subcategory of $\fg$-modules
consisting of objects $M$ with  diagonal action of $h$ with
the eigenvalues lying in $b+\mathbb{Z}$.
Since $\fsl_2$ acts locally finitely on $\fg$,  the operator $ad\ h$ on $\fg$
is diagonal with integral eigenvalues and thus any
indecomposable module with a diagonal action of $h$ lies in
$\mathcal M(b)$ for some $b$.
We denote by $\mathcal M(b)^e$ the full subcategory of
  $\mathcal M(b)$ consisting of objects with locally nilpotent action of
  $e$.

\subsubsection{}
\begin{lem}{lemsl2a}
Take $\fg=\fsl_2$ or $\fg=\mathfrak{osp}(1|2)$ and
$b\not\in\mathbb{Z}$. A module $N\in \mathcal{M}(b)^e$ is a direct sum
of simple Verma modules.
\end{lem}
\begin{proof}
Take $v\in N$.
The submodule
generated by $v$ is a cyclic module with a locally nilpotent action of $e$;
since $b\not\in\mathbb{Z}$, this module is a simple Verma module.
Therefore $N$ is a sum of simple modules, so
$N$ is completely reducible by~\cite{Lang}, Ch. XVII, Sect. 2.
\end{proof}

\subsubsection{}
\begin{prop}{proptwEnright}

(i) Let $\fg'$ be a subalgebra of $\fg$ containing
  $\fsl_2$ and $\cC'_{a}$ be an analogue of $\cC_{a}$
  for $\fg'$. Then $\Res^{\fg}_{\fg'}\cC_{a}\cong\cC'_{a}\Res^{\fg}_{\fg'}$.

(ii) The restriction of the Enright functor $\cC_a$ to  $\mathcal{M}(a)$ defines an additive left exact functor
$\mathcal{M}(a)\to \mathcal M(-a)^e$.

(iii) Take $V\in \cM(b)^e$.
 If the module $\dD_a(V)$ has a non-zero
subquotient with locally nilpotent action of $e$, then $a-b\in\mathbb{Z}$
or $a\in\mathbb{Z}$.

(iv) For $N\in \mathcal{M}(0)^e$ the map $v\mapsto 1\otimes v$
gives a homomorphism $N\to \cC_0(N)$ with the kernel $N^f$.

(v) Take $a\not\in\mathbb{Z}$ and $N\in  \mathcal{M}(a)^e$. For every
$v\in N$ there exists $s>0$ such that $f^{a+j}\otimes v\in \cC_a(N)$
for every $j\geq s$.
\end{prop}
\begin{proof}
Both the Zuckerman functor and the twisted localization functor
commute with $\Res_{\fg'}$ (see~\S~\ref{propertydD}), so (i) follows.
  Since  the Zuckerman functor
 is additive and left exact,
$\cC_a$ is an additive left exact functor. From~(\ref{sl2formula})
it follows that for $M\in \mathcal M(a)$ one has $\cC_a(M)\in \mathcal M(-a)$, which gives (ii). {By (i)
it suffices to prove (iii) for $\fg=\mathfrak{sl}_2$.
  Since $\dD_a$ is exact, it suffices to prove the statement for simple $V$. Any simple
  $\fsl_2$-module
  $V\in  \mathcal{M}(a)^e$ is either finite-dimensional and then  $\dD_a(V)=0$
  or a Verma module. In the latter case the statement follows from~\Lem{lemtwVerma} (i).
  This gives (iii).}

By~\S~\ref{dD}, $N^f$ is the kernel of the map $N\to N[f^{-1}]$
and this gives (iv).

For (v) take $v'\in N$. Using~\Lem{lemsl2a}
we may (and will) assume that  $v'=f^jv$, where $ev=0$ and $hv=bv$
for $b\in a+\mathbb{Z}$. By~\Lem{lemtwVerma} (i), $f^{b+1}\otimes v \in \cC_a(N)$,
so $f^{a+j}\otimes v' \in \cC_a(N)$ if $a+j\geq b+1$.
\end{proof}

\subsubsection{Remark}\label{purim}
By (iii), the restriction of $\cC_a$ to $\cM(b)^e$ is non-zero only for
 $a-b\not\in\mathbb{Z}$ or $a\in\mathbb{Z}$. In the former case, by (ii), we have
 the functor $\cM(a)^e\to \cM(-a)^e$; from now on we denote this functor by
 $\cC_a$. In the latter case $\cC_a$ restricts to
 the endofunctor $\cM(b)^e\to \cM(b)^e$; furthermore, for $b\not\in\mathbb{Z}$
 this functor is isomorphic to the identity functor by~\Lem{lemtwVerma} (i).

  \subsubsection{Example: $\fsl_2$}
Take $\fg=\fsl_2$. If $a\not\in\mathbb{Z}$, all indecomposable modules in the category $\cM(a)$  are Verma modules
and $\cC_a(M(a))=M(-2-a)$ by~\Lem{lemtwVerma}.

{By~\cite{BGG} }
the category $\cM(0)$ contains
five types of indecomposable modules: for $i\in\mathbb{Z}_{\geq 0}$ we have a
finite-dimensional module
$L(i)$, a Verma modules $M(i)$,  its dual module $M^{\#}(i)$, a simple Verma module $M(-1-i)$,
and a self-dual module of length three $N(i)$ with the exact sequence
$$0\to M(i)\to N(i)\to M(-2-i)\to 0.$$
Using~\Lem{lemtwVerma} we obtain
$$\cC_0(L(i))=0, \ \cC_0(M(i))=\cC_0(M^{\#}(i))= \cC_0(M(-2-i))=M(i),\ \cC_0(N(i))=N(i).$$

\subsubsection{}
\begin{thm}{propid}
(i) Consider the restriction $\cC_a:\mathcal M(a)^e\to\mathcal M(-a)^e$.
  There exists a morphism of functors
  $\operatorname{Id}_{\mathcal M(a)^e}\to \cC_{-a}\cC_{a}$.

  (ii) For $a\not\in\mathbb{Z}$ the functor $\cC_a:\mathcal M(a)^e\to\mathcal M(-a)^e$
  is an equivalence of categories.

  (iii) Let $N_1,N_2\in \mathcal M(0)^e$ be such that $N_1^f=0$.
Then the natural map $\Hom_{\fg}(N_1,N_2)\to \Hom_{\fg}(\cC_0(N_1),\cC_0(N_2))$ is injective and hence
$$\dim \Hom_{\fg}(N_1,N_2)\leq \dim \Hom_{\fg}(\cC_0(N_1),\cC_0(N_2)).$$
  \end{thm}
\begin{proof}  First we claim that for any $M\in\mathcal M(a)^e$ the natural morphism $\dD_0\cC_aM\to \dD_0\dD_aM$ is an isomorphism.
    Indeed, consider the exact sequence $0\to \cC_aM\to \dD_aM\to X\to 0.$ Then $X^f=X$ because otherwise $\dD_aM$ has an $\fsl_2$-subquotient
    with free action of $f$ and $e$ which is not the case by ~\Lem{lemtwVerma} (i). Therefore $\dD_0X=0$ and we have an isomorphism of functors
    $\dD_{-a}\cC_a\simeq \dD_{-a}\dD_a$. Hence there is a canonical morphism $M\to \dD_{-a}\cC_aM$. On the other hand $\Hom_{\fg}(M,X)=\Hom_{\fg}(M,\Gamma_e X)$.  This implies existence of a canonical morphism $\Phi_M: M\to \cC_{-a}\cC_aM$ and (i) is proven.

    For (ii) we have to check that $\Phi_M$ is an isomorphism. From~\Lem{lemtwVerma} (i) it follows that for the case $\fg=\fsl_2$.
    In general it is a consequence of~\Prop{proptwEnright} (i) and ~\Lem{lemsl2a}.

  For (iii) recall that the map $v\mapsto 1\otimes v$ induces an embedding
  $N_1\to \cC(N_1)$. For $\psi\in \Hom_{\fg}(N_1,N_2)$
  the restriction of $\cC(\psi)$ to $N_1\subset \cC(N_1)$ is equal to $\psi$.
  This implies the required inequality.
  \end{proof}

\subsection{Changing characters for $a\not\in\mathbb{Z}$}
In this subsection $\fh\subset \fg$ is an abelian Lie subalgebra containing $h$
such that $[h',e]=\alpha(h')e$, $[h',f]=-\alpha(h')f$ for some $\alpha\in\fh^*$ and
every $h'\in\fh$. We denote the fixed copy of  $\fsl_2$ (the span of $f,h,e$) by $\fsl_2(\alpha)$.
We define $N_{\nu}$ and $\ch_{\epsilon} N$ as in~\S~\ref{Nnu};
we set $r_{\alpha}\nu:=\nu-\nu(h)\alpha$ for $\nu\in\fh^*$ and
$r_{\alpha}\bigl(\sum_{\nu\in\fh^*} a_{\nu}e^{\nu}\bigr):=
\sum_{\nu\in\fh^*} a_{\nu}e^{r_{\alpha}(\nu)}$.

\subsubsection{}
\begin{thm}{thmchar}
Let $\fsl_2=\fsl_2(\alpha)$.
Take $a\not\in\mathbb{Z}$ and $N\in\cM(a)^e$.
If $N$ and $\cC_a(N)$ admit the $\fh$-characters, then

 (i)
$\ \  \ e^{\frac{\alpha}{2}}(1-e^{-\alpha})
\ch_{\epsilon} \cC_a(N)=r_{\alpha}\bigl(e^{\frac{\alpha}{2}}(1-e^{-{\alpha}})\ch_{\epsilon} N\bigr)$.

(ii) If $\fg$ contains $\mathfrak{osp}(1|2)$ with $\mathfrak{osp}(1|2)_{\ol{0}}=\fsl_2(\alpha)$, then
$$e^{\frac{\alpha}{4}}(1-\epsilon e^{-\frac{\alpha}{2}})
\ch_{\epsilon}\cC_a(N)=\epsilon r_{\alpha}\bigl(e^{\frac{\alpha}{4}}(1-\epsilon e^{-\frac{\alpha}{2}})\ch_{\epsilon} N\bigr).$$
\end{thm}
\begin{proof}
We set $\ft:=\fsl_2$ for (i) and $\ft:=\mathfrak{osp}(1|2)$
for (ii).

By~\Prop{proptwEnright} (i)
 $\cC_a$ commutes with the restriction functor $Res^{\fg}_{\ft+\fh}$.
Thus it is enough to consider the case $\fg=\ft+\fh$. One has $\fg=\fh'\times\ft$,
where $\fh'$ is the centralizer of $\ft$. The module $\Res^{\fg}_{\ft} N$
is semisimple, see~\Lem{lemsl2a}. Therefore
$$N=\bigoplus_{b\in a+\mathbb{Z}} \Hom_{\ft}(M_{\ft}(b), N)\boxtimes M_{\ft}(b)$$
and
$$\cC_a(N)=\bigoplus_{b\in a+\mathbb{Z}} \Hom_{\ft}(M_{\ft}(b), N)\boxtimes \cC_a(M_{\ft}(b)).$$

Thus it is enough to verify the assertion for {$\fg=\ft$ and  $N=M_{\ft}(b)=M_{\ft}(x\alpha)$ ($x=\frac{b}{2}$)}.

Consider the case $\ft=\fsl_2$.
 By~\Lem{lemtwVerma},
$\cC_a(M_{\fsl_2}(x\alpha))=M_{\fsl_2}(-(x+1)\alpha)$, so
$$\begin{array}{rl}
(1-e^{-\alpha})
\ch_{\epsilon} \cC_a(M_{\fsl_2}(x\alpha)))&=(1-e^{-\alpha})\ch_{\epsilon} M_{\fsl_2}(-(x+1)\alpha)
=e^{-(x+1)\alpha}\\
&=e^{-\alpha}
r_{\alpha}\bigl((1-e^{-\alpha})\ch_{\epsilon} M_{\fsl_2}(x\alpha)
\bigr)\end{array}$$
as required.

Consider the case $\ft=\mathfrak{osp}(1|2)$.
One has
 $\ch_{\epsilon} M_{\mathfrak{osp}(1|2)}(y\alpha)=(1-\epsilon e^{-\alpha/2})^{-1}e^{y\alpha}$
for any $y\in\mathbb{C}$.
 By~\Lem{lemtwVerma},
$\cC_a(M_{\mathfrak{osp}(1|2)}(x\alpha))=\Pi(M_{\mathfrak{osp}(1|2)}(-(x+1/2)\alpha))$.
We obtain
$$\begin{array}{l}(1-\epsilon e^{-\frac{\alpha}{2}})
\ch_{\epsilon} \cC_a(M_{\mathfrak{osp}(1|2)}
(x\alpha)))=\epsilon(1-\epsilon e^{-\frac{\alpha}{2}})
\ch_{\epsilon} M_{\mathfrak{osp}(1|2)}
(-(x+\frac{1}{2})\alpha)\\
\ \ =\epsilon e^{-(x+\frac{1}{2})\alpha}
=\epsilon e^{-\frac{\alpha}{2}}r_{\alpha}\bigl((1-\epsilon e^{-\frac{\alpha}{2}})\ch_{\epsilon} M_{\mathfrak{osp}(1|2)}(x\alpha)\bigr)\end{array}$$
as required.
\end{proof}

\subsubsection{Example}\label{exa1} This example illustrates that $\cC_aM$ may not admit $\fh$-character even if $M$ admits it.
Take $\fg=\fsl_2$ and $\fh=\mathbb{C}h$. The module
$$N:=\displaystyle\bigoplus_{i=0}^{\infty}  M_{\fsl_2}(a-2i)$$
admits $\fh$-character. For $a\not\in\mathbb{Z}$ one has
$$\cC_a(N)\cong \displaystyle\bigoplus_{i=0}^{\infty}  M_{\fsl_2}(-a+2i-2)$$
and $\dim \cC_a(N)_{-a}=\infty$, so $\cC_a(N)$ does not admit $\fh$-character.

\subsubsection{Example}\label{exasl3}
Take $\fg=\fsl_3$ with the simple roots $\alpha_1,\alpha_2$
and consider
the $\fsl_2$-copy $\fsl_2(\alpha_1+\alpha_2)\subset\fg$ corresponding to the root $\alpha_1+\alpha_2$. Set $\fg':=\fh+\fsl_2$. Let $v_{\lambda}$
be the highest weight vector of a Verma $\fg$-module $M(\lambda)$.

Since $e\in\fg_{\alpha_1+\alpha_2}$
and  $e(M(\lambda)_{\lambda-i\alpha_1})\subset M(\lambda)_{\lambda-(i-1)\alpha_1
+\alpha_2}=0$,
the module $M(\lambda)$ contains
$\fg'$-submodules $M_{\fg'}(\lambda-i\alpha_1)$ for $i\geq 0$. Moreover, by PBW theorem,
the sum of these submodules is direct, so
$M(\lambda)$ contains
$\oplus_{i=0}^{\infty} M_{\fg'}(\lambda-i\alpha_1)$.
Assume that $a:=\lambda(h)\not\in\mathbb{Z}$. From~\Lem{lemtwVerma}
we conclude that $\cC_a(M_{\fg'}(\nu))=M_{\fg'}(r_{\alpha_1+\alpha_2}\nu-
(\alpha_1+\alpha_2))$.
Therefore $\cC_a(M(\lambda))$ contains a $\fg'$-submodule
$$N':=\cC_a\bigl(\displaystyle\bigoplus_{i=0}^{\infty} M_{\fg'}(\lambda-i\alpha_1)\bigr)=
\displaystyle\bigoplus_{i=0}^{\infty} M_{\fg'}(r_{\alpha_1+\alpha_2}\lambda-\alpha_1+(i-1)\alpha_2).$$

In particular, $\cC_a(M(\lambda))$  does not lie in the category $\CO$ for $\fsl_3$.
Moreover,   $M(\lambda)$ admits  $\fh'$-character and
$\cC_a(M(\lambda))$  does not admit  $\fh'$-character for
$\fh':=\mathbb{C}h$.

\section{Enright functor for Kac--Moody superalgebras}\label{En2}
Let $\fg$ be a Kac--Moody Lie superalgebra.
Retain notations of~\S~\ref{contra}. To simplify notations we
identify the modules which differ by parity change.

In this section
 $\alpha\in\Pi_{pr}$
and $\fsl_2=\fsl_2(\alpha)$ is  such that
$$e\in\fg_{\alpha}, f\in\fg_{-\alpha},\
h=\alpha^{\vee}.$$

If $\frac{\alpha}{2}\in\Delta$, then
$\fg$ contains a copy of $\mathfrak{osp}(1|2)$ with $\mathfrak{osp}(1|2)_{\ol{0}}=\fsl_2(\alpha)$
and $\mathfrak{osp}(1|2)_{\ol{1}}=\fg_{-\alpha/2}+\fg_{\alpha/2}$.

Retain notation of~\ref{COK}.
We denote by  $\CO^{fin}(a)$ the intersection of $\CO^{fin}$ and $\cM(a)$,
i.e. the full subcategory of $\CO$ with finitely generated objects, where
the eigenvalues of $h$ lie in $a+\mathbb{Z}$.

In~\Thm{thmeqO} we will show that for $a\not\in\mathbb{Z}$ the functor
 $\cC_a$ gives an equivalence of categories
${\CO}^{fin}(a)\iso{\CO}^{fin}(-a)$. The assumption  $\alpha\in\Pi_{pr}$ is crucial. Indeed Example~\ref{exasl3}
illustrates that if $\alpha\not\in \Pi_{pr}$, then
$\cC_a(M(\lambda))$ may be not  in the category ${\CO}(-a)$;
example~\ref{exa1} shows that the character formula in~\Thm{thmeqO} (iii)
does not hold for the category $\CO$.

\subsection{}\label{sub1}
\begin{prop}{thmVerma}
Assume that $\alpha\in\Sigma$ or $\alpha/2\in\Sigma$.

(i) Let $N$ be a module in $\cM(a)^e$ and let $v\in N_{\nu}$ be a singular vector. The vector
 $$v':=\left\{\begin{array}{ll}
    f^{\nu(h)+1}\otimes v\text{ if }\alpha\in\Sigma,\\
   f^{\nu(h)}\otimes Fv\text{ if }\alpha/2\in\Sigma
    \end{array}   \right. $$
has weight $r_{\alpha}.\nu$ and    is a singular vector in $\cC_a(N)$.

(ii) One has $\cC_a(M(\lambda))=M(r_{\alpha}.\lambda)$ if
$a:=\lambda(\alpha^{\vee})$ is not integral.

(iii) Let $a\not\in\mathbb{Z}$.
If $N\in \CO^{fin}(a)$ is generated by a singular vector $v$ of weight $\nu$, then
$\cC_a(N)\in \CO^{fin}(-a)$ is generated by a singular vector $v'$ described
as in (i).
\end{prop}
\begin{proof}
  \Lem{lemtwVerma} implies (i). For (ii) take
$a\not\in\mathbb{Z}$.
By ~\Thm{propid} $\cC_a: \cM(a)^e\to \cM(-a)^e$
is equivalence of categories.
Now (i) gives
   $$\Hom_{\fg}(M(r_\alpha.\lambda),\cC_a(M(\lambda))\not=0.$$
   Let $\varphi\in \Hom_{\fg}(M(r_\alpha.\lambda),\cC_a(M(\lambda))$ and $\varphi\neq 0$. Then $\cC_{-a}\varphi(M(r_\alpha.\lambda))$ is a submodule of $M(\lambda)$
   which by (i) has a singular vector of weight $\lambda$. Hence we have an isomorphism $\cC_{-a}\varphi(M(r_\alpha.\lambda))\simeq M(\lambda)$ which implies
   an isomorphism
$$\cC_a\cC_{-a}\varphi(M(r_\alpha.\lambda))\simeq \varphi(M(r_\alpha.\lambda))\simeq \cC_aM(\lambda).$$
Now (iii) follows from (i) and (ii).
 \end{proof}

\subsection{}\label{sub3}
\begin{thm}{thmeqO}
Let $\alpha$ or $\alpha/2\in\Sigma$
and $a\not\in\mathbb{Z}$.

(i) The Enright functor
$\cC_a$ induces an equivalence of categories
${\CO}^{fin}(a)\iso {\CO}^{fin}(-a)$.

(ii) If $\lambda(h)\in\mathbb{Z}+a$, then
$\cC_a(L(\lambda))=L(r_{\alpha}.\lambda)$.

(iii) If $N\in {\CO}\cap\cM(a)$ is such that $\cC_a(N)\in {\CO}$, then
$Re^{\rho}\ch \cC_a(N)=r_{\alpha}\bigl(Re^{\rho}\ch N\bigr)$.
In particular,  the formula holds for $N\in {\CO}^{fin}$.
\end{thm}
\begin{proof}
First, let us verify that $\cC_a(N)\in {\CO}^{fin}(-a)$  for any
$N\in {\CO}^{fin}(a)$.
If $N$ is a highest weight module, this follows from~\Prop{thmVerma}.
The general case follows from~\Lem{lemtildeCO} and exactness of $\cC_a$.

Since $\cC_a$ induces an equivalence of categories $\cM(a)^e\iso\cM(-a)^e$
and $\cC_a({\CO}^{fin}(a))\subset {\CO}^{fin}(-a)$, we obtain (i).

Since $L(\lambda)$ is a unique simple quotient of $M(\lambda)$,
(ii) follows from (i) and~\Prop{thmVerma} (i).

For (iii) we introduce {$\beta, R_{\alpha}$ as in the proof of~\Lem{lemR};
 recall that $r_{\alpha}R_{\alpha}=R_{\alpha}$ and
$R=R_{\alpha}(1-e^{-\beta})$.}
Using~\Thm{thmchar} we get
$$\begin{array}{ll}
Re^{\frac{\beta}{2}}\ch \cC_a(N)=&R_{\alpha}e^{\frac{\beta}{2}}(1-e^{-\beta})\ch \cC_a(N)=
R_{\alpha}r_{\alpha}\bigl(e^{\frac{\beta}{2}}(1-e^{-\beta})\ch N\bigr)\\
&=r_{\alpha}\bigl(R_{\alpha}e^{\frac{\beta}{2}}
(1-e^{-\beta})\ch N\bigr)=r_{\alpha}\bigl(Re^{\frac{\beta}{2}}\ch N\bigr).\end{array}$$
Since $r_{\alpha}(\rho-\frac{\beta}{2})=\rho-\frac{\beta}{2}$
this gives the required formula.
\end{proof}

\subsection{Remark}
We denote by $\tilde{\CO}^{fin}$ the full category of $\fg$-modules $N$
with the finitely generated objects satisfying the following property:
$\fh$ acts locally finitely and for any $v\in N$ one has
$\dim U(\fn^+)v<\infty$.

The statements (i), (iii) of~\Thm{thmeqO} hold for the category $\tilde{\CO}^{fin}$.
Similarly to~\ref{Ofininf} we have the following lemma.

\subsubsection{}
\begin{lem}{lemtildeCO}
A $\fg$-module $N$ lies in $\tilde{\CO}^{fin}$ if and only if $N$ admits
a  finite filtration whose quotients are
highest weight modules.
\end{lem}
\begin{proof}
Note that for a short exact sequence
$$0\to N_1\to N\to N_2\to 0$$
with  $N_1,N_2\in  \tilde{\CO}^{fin}$
one has $N\in  \tilde{\CO}^{fin}$. As a result, ``if''
follows by induction on the length of the filtration.

One readily sees that every  cyclic module in $\tilde{\CO}^{fin}$,
which is
 generated by a weight vector,
admits a  finite filtration with the quotients which are
highest weight modules. Since any
 module in  $\tilde{\CO}^{fin}$ is finitely generated, it
admits such filtration.
\end{proof}

\subsection{}
\begin{prop}{thmaint}
Let $\alpha$ or $\alpha/2\in\Sigma$.
Let $\nu\in\fh^*$ be such that $\nu(h)\in\mathbb{Z}$.
One has $\cC_0(M(\nu))=M(\nu')$, where $\nu'=\nu$ if $\nu(h)\geq 0$ and
$\nu':=r_{\alpha}.\nu$ otherwise.
\end{prop}
\begin{proof}
Set $\ft:=\fsl_2$ if $\alpha\in\Sigma$ and
$\ft:=\mathfrak{osp}(1|2)$ if $\alpha/2\in\Sigma$.

First, consider the case $\fg=\ft$. By~\Lem{lemtwVerma}, $\cC_0(M(\nu))=M(\nu')$
and $\dD_0(M(\nu))/M(\nu')\cong M_{\fb^-}(\nu'+x\alpha)$,
where $x=1$ (resp., $x=1/2$) for $\fsl_2$ (resp., for $\mathfrak{osp}(1|2)$);
the module $M_{\fb^-}(\nu'+x\alpha)$ is simple.
We set $M':=M_{\fb^-}(\nu'+x\alpha)$.

Now consider the case $\fg\not=\ft$.
Consider the parabolic subalgebra $\fp:=\ft+\fn^+$.
One has
$$M(\nu)=Ind_{\fp}^{\fg}\, M_{\ft}(a),\text{ where }a:=\nu(h)$$
and the $\fp$-action on $M_{\ft}(a)$ is given by the zero action of the radical of $\fp$.
The functors $Ind$ and $\dD$ commute (see, for example,~\cite{Gr}, Lem. 2.4), that is
$$\dD_0(M(\nu))=Ind_{\fp}^{\fg}\ \dD_0(M_{\ft}(a)).$$
One has $\fg=\fm^-\oplus\fp$, where
 $\fm^-$ is the radical of $\ft+\fn^-$.
We obtain the following exact sequence
$$0\to M(\nu')\to \dD_0(M(\nu))\to Ind_{\fp}^{\fg} M'\to 0$$
where $M'$ is the $\ft$-module introduced above with the zero action
of the radical of $\fp$. Since $(M')^e=0$ and $[e,\fm^-]\subset\fm^-$, one has
$$(Ind_{\fp}^{\fg}\, M')^e=(U(\fm^-)\otimes M')^e=0$$
so $\cC_0(M(\nu))=M(\nu')$ as required.
\end{proof}

\subsection{Chain of Enright functors}\label{chain}
Take $\alpha\in\Pi_{pr}$
 and consider
$\fsl_2=\fsl_2(\alpha)$. Since $\CO^{fin}=\oplus_{a\in\mathbb{C}/\mathbb{Z}}\CO^{fin}(a)$ we can define the Enright functor
$\cC^{\alpha}$  on $\CO^{fin}$  as the direct sum {of the functors} $\cC_a:\CO^{fin}(a)\to \CO^{fin}(-a)$ for all $a\in\mathbb C/\mathbb Z$.
For a sequence of roots $\alpha_1,\ldots,\alpha_s\in\Pi_{pr}$
we set
$$\cC^{\alpha_1,\ldots,\alpha_s}:=\cC^{{\alpha_s}}\,\circ
\cdots\,\circ \cC^{{\alpha_1}}.$$

\subsubsection{}
We retain notations of~\S~\ref{Deltalambda},\ref{COK}. Recall that
the category  $\CO^{fin}$ is the direct sum of
the categories $\CO_{\Delta'}$ for $\Delta'$ being root subsystems
of $W\Pi_{pr}$. Using~\Thm{thmeqO}
we obtain the following corollary.

\subsubsection{}
\begin{cor}{corOnu}
Let $\fg$ be a Kac--Moody Lie superalgebra.
Take $\alpha_1,\ldots,\alpha_s\in \Pi_{pr}$ such that
$\alpha_i\in\Sigma$ or $\frac{\alpha_i}{2}\in\Sigma$ for every $i$.
Let $\Delta'\subset W\Pi_{pr}$ be a root subsystem which is
 $(\alpha_1,\ldots,\alpha_s)\,$-friendly. Set
 $w:=r_{\alpha_s}
\ldots r_{\alpha_1}$.

(i) The functor  $\cC:=\cC^{\alpha_1,\ldots,\alpha_s}$ provides an equivalence of categories
$\CO^{fin}_{\Delta'}\to \CO^{fin}_{w\Delta'}$ and
$$\cC(M(\lambda))=M(w.\lambda),\ \ \cC(L(\lambda))=L(w.\lambda)$$
if $\Delta(\lambda)=\Delta'$.

(ii) For any $N\in \CO^{fin}_{\Delta'}$ one has
$\ \ Re^{\rho}\ch\, \cC(N)=w\bigl(Re^{\rho}\ch N\bigr)$.
\end{cor}

\subsubsection{}
\begin{cor}{corOnu1}
  Let $\fg$ be a Kac--Moody Lie superalgebra. Take $\nu\in\fh^*$
and   $\beta\in\Pi(\nu)$.
By~\Lem{lemwmin} there exists $w=r_{\alpha_s}
\ldots r_{\alpha_1}\in W$ such that
  $\nu$ is $(\alpha_1,\ldots,\alpha_s)$-friendly and $w\beta\in\Pi_{pr}$.
   Take $\cC:=\cC^{\alpha_1,\ldots,\alpha_s}$.

  (i) The functor  $\cC$ provides an equivalence of
  categories $\CO^{fin}_{\Delta(\nu)}\iso \CO^{fin}_{w\Delta(\nu)}$.

  (ii)  If $\nu+\rho$ is typical, then
  $\cC(M(\nu))=M(w.\nu)$,  $\cC(L(\nu))=L(w.\nu)$.

(iii) If $\fg$ is symmetrizable and $\nu+\rho$ is typical, then
$$Re^{\rho}\ch L(w.\nu)=w\bigl(Re^{\rho}\ch L(\nu)).$$
\end{cor}
\begin{proof}
  Applying~\Thm{thmeqO} (i) and changing the base by odd reflections, we obtain (i). Now (ii)
  follows from
(i) and~\S~\ref{lambdalambda'}. Combining (ii),~\Prop{supptyp}
and~\Cor{corOnu} (ii), we obtain (ii).
\end{proof}

\subsection{}
\begin{prop}{cortyp}
  Let $\fg$ be a symmetrizable Kac-Moody superalgebra
  and $\lambda\in\fh^*$ be such that $\lambda+\rho$ is  typical.

  (i) For $\beta\in\Pi(\lambda)$
 $$
   r_{\beta}\bigl(Re^{\rho}\ch L(\lambda)\bigr)=-Re^{\rho}\ch
   L(\lambda)\ \Longleftrightarrow\ (\lambda+\rho)(\beta^{\vee})>0.$$

 (ii) If $\lambda+\rho$ is non-critical  and
$(\lambda+\rho)(\alpha^{\vee})>0$
for each $\alpha\in\Pi(\lambda)$, then

$$Re^{\rho}\ch L(\lambda)=
\sum_{w\in {W}(\lambda)} (-1)^{l(w)} e^{w(\lambda+\rho)}.$$
\end{prop}
\begin{proof}
For (i) consider the formula
\begin{equation}\label{R-R}
   r_{\beta}\bigl(Re^{\rho}\ch L(\lambda)\bigr)=-Re^{\rho}\ch
   L(\lambda)\end{equation}

If $\lambda$ satisfies~(\ref{R-R}),
then $r_{\beta}(\lambda+\rho)\in (\lambda+\rho-Q^+)\setminus\{\lambda+\rho\}$, so
$(\lambda+\rho)(\beta^{\vee})>0$.

Now assume that $(\lambda+\rho)(\beta^{\vee})>0$.
By~\Prop{supptyp} $Re^{\rho}\ch L$ does not depend on
$\Sigma$.

For $\beta\in \Pi_{pr}$
  we take $\Sigma$ which contains $\beta$ or $\beta/2$;
in this case $\fg_{-\beta}$ acts nilpotently on the singular vector
of $L(\lambda)$ and thus locally
nilpotently on $L(\lambda)$ (see~\cite{Kbook2}, Lem. 3.4). This gives~(\ref{R-R}).

Now consider any $\beta\in\Pi(\lambda)$.
Taking $w$ as in~\Cor{corOnu1} we obtain
$$r_\beta \left(Re^{\rho}\ch L(\lambda)\right)=
r_\beta w^{-1}\bigl( Re^{\rho}\ch L(w.\lambda) \bigr)=w^{-1}r_{w\beta} \bigl( Re^{\rho}\ch L(w.\lambda) \bigr).$$
Now~(\ref{R-R}) for $\beta$ follows from~(\ref{R-R}) for $w\beta$
(since $w\beta\in\Pi_{pr}$). This establishes (i).

For (ii) recall that
$$Re^{\rho}\ch L(\lambda)=
\sum_{w\in {W}(\alpha)} a_w e^{w(\lambda+\rho)},$$
by~(\ref{suppN}).
By~\Prop{eq>0} (iii) the assumption on $\lambda$ gives
  $\Stab_{{W}(\lambda)} (\lambda+\rho)=Id$.
By~\S~\ref{Deltalambda}, ${W}(\lambda)$ is generated by
 $r_{\alpha}$ with  $\alpha\in\Pi(\lambda)$. Using~(\ref{R-R}) we
 obtain (ii).
 \end{proof}

\begin{rem}{}
\Prop{cortyp} (ii) was proven in~\cite{GK}, Thm. 11.1.2.
\end{rem}

\subsection{}\label{sigmapr} Let us fix $\Sigma\in\cB$. Let
$$\Sigma_{pr}:=\{\alpha\in \Pi_{pr}\,|\,\alpha\
\text{or}\ \alpha/2\in\Sigma\}$$
and let $W_\Sigma$ be the subgroup
of $W$ generated by $r_{\alpha}$ for all $\alpha\in \Sigma_{pr}$. Clearly, $W_{\Sigma}\Sigma_{pr}$ is a root subsystem and
$\Delta(\lambda)\cap W_{\Sigma}\Sigma_{pr}$ is a root subsystem for any weight $\lambda$. We denote by $\Sigma(\lambda)$  the base of
$\Delta(\lambda)\cap W_{\Sigma}\Sigma_{pr}$  in the sense
of~\ref{Deltalambda}.

\subsubsection{}
\begin{thm}{thmchN}
Let $\lambda\in\fh^*$ and $\beta\in \Sigma(\lambda)$
be such that $(\lambda+\rho)(\beta^{\vee})>0$.

(i) One has $\dim \Hom_{\fg}(M(r_{\beta}.\lambda), M(\lambda))=1$
and $M(r_{\beta}.\lambda)$ is a submodule of $M(\lambda)$.

(ii)  For any subquotient $N$
of $M(\lambda)/M(r_{\beta}.\lambda)$ one has
$r_{\beta}\bigr(Re^{\rho} ch N)=-Re^{\rho} ch N$.

(iii) $[M(\lambda)/M(r_{\beta}.\lambda):L(\nu)]\not=0$ implies
$(\nu+\rho,\beta)>0$.
\end{thm}
\begin{proof}
Consider the case $\beta\in\Sigma_{pr}$.
Then (i) follows from the corresponding assertions
for $\mathfrak{sl}_2,\mathfrak{osp}(1|2)$. By~\cite{Kbook2}, Lem. 3.4,
 $\fg_{-\beta}$ acts locally nilpotently
on $M(\lambda)/M(r_{\beta}.\lambda)$ and thus on any subquotient $N$
of  $M(\lambda)/M(r_{\beta}.\lambda)$; this gives (iii) and
 $\ch N=r_{\beta}\ch N$.
By~\Lem{lemR}, $Re^{\rho}=-r_{\beta}(Re^{\rho})$, which implies (ii).

In the general case $\beta\in \Sigma(\lambda)$  using~\Cor{corOnu1}
for the Kac-Moody algebra corresponding to
$\Sigma_{pr}$ we obtain $w\in W_\Sigma$ such that $w\beta\in \Sigma_{pr}$ and $\lambda$ is $w$-friendly.
~\Cor{corOnu} implies
$$\begin{array}{l}
\cC(M(\lambda))=M(w.\lambda),\ \
\cC(M(r_{\beta}.\lambda))=M(wr_{\beta}.\lambda)=M(r_{w\beta}(w.\lambda)),\\
(w.\lambda+\rho, (w\beta)^{\vee})=(\lambda+\rho,\beta^{\vee})\end{array}$$
and $w\beta\in \Sigma(w.\lambda)\cap \Sigma$.
Since
  $w\beta\in \Sigma$, the assertions (i)---(iii) hold for
 the pair  $(w.\lambda,w\beta)$. Since $\cC$ is an equivalence of categories
(i)---(iii) holds for the pair $(\lambda,\beta)$.
 \end{proof}

\subsubsection{}
\begin{cor}{corchN}
For $\lambda\in\fh^*$ and  $\beta\in \Sigma(\lambda)$ one has
$$r_{\beta}\bigl(Re^{\rho}\ch L(\lambda)\bigr)=-Re^{\rho}\ch L(\lambda)\
\Longleftrightarrow\
(\lambda+\rho)(\beta^{\vee})>0.$$
\end{cor}
\begin{proof}
The implication $\ \Longleftarrow\ $ follows from~\Thm{thmchN} above.

Now assume that
$r_{\beta}\bigl(Re^{\rho}\ch L(\lambda)\bigr)=-Re^{\rho}\ch L(\lambda)$.
One has $r_{\beta}(\lambda+\rho)=\lambda+\rho-a\beta$, where
$a:=(\lambda+\rho)(\beta^{\vee})$.
Since
$$(\lambda+\rho)\in  \supp \bigl(Re^{\rho}\ch L(\lambda)\bigr)\subset
(\lambda+\rho-Q^+)$$
one has $a\not=0$ and $a\beta\in Q^+$.
We have $w\beta\in\Sigma$ for some $w\in W$, which gives
$a(w\beta)\in wQ^+\subset\mathbb{Z}\Sigma$, so $a\in\mathbb{Z}$.
If $a<0$, then $a\beta\in -Q^+$.
Hence $a>0$ as required.
\end{proof}

\section{Snowflake modules for non-isotropic type}
In this section $\fg$ is a symmetrizable Kac--Moody superalgebra of
non-isotropic type (i.e., $\fg$ does not have isotropic simple roots).
{Many results of this section (\Prop{propsnow}, \Cor{corsnow}, \Thm{thmcomred} (ii) and~\Cor{snowcor})
were proven, under
a certian additional assumption,   in~\cite{KW-1} using translation functors.}
By~\Lem{lemR},
 $w(Re^{\rho})=(-1)^{l(w)} Re^{\rho}$ for $w\in W$.
{Recall that $\CO_{\Delta'}$ is the full category of $\CO$ with $\Delta(N)=\Delta'$,
see~\ref{ODelta}.}

\subsection{}
\begin{defn}{defnsnowflake}
 We call  a $\fg$-module $N\in \CO$
 {\em snowflake } if $N\in \CO_{\Delta'}$ for some $\Delta'$ and
$$\forall\alpha\in\Delta'\ \  \ \ \ \ r_{\alpha}\bigl(Re^{\rho}\ch N\bigr)=-Re^{\rho}\ch N.$$
\end{defn}

\subsubsection{Remarks}\label{remsnowflake}
Since the group $W(N)$ is generated by $r_{\alpha}$ with
 $\alpha\in \Delta(N)$, the above formula can be rewritten as
$w\bigl(Re^{\rho}\ch N\bigr)=(-1)^{l(w)}Re^{\rho}\ch N$ for $w\in W(N)$
(we call this property $W(N)$-skew-invariance); by~\Lem{lemPilambda}
this property is equivalent to
$r_{\alpha}(Re^{\rho}\ch L(\lambda))=-Re^{\rho}\ch L(\lambda)$
for every $\alpha\in\Pi(N)$.

\subsection{}
\begin{thm}{thmsnow}
The following conditions are equivalent

\hspace{1cm} $L(\lambda)$ is snowflake;

  \hspace{1cm} $(\lambda+\rho)(\alpha^{\vee})>0$ for each $\alpha\in\Pi(\lambda)$.
\end{thm}
\begin{proof} An immediate consequence of ~\Cor{corchN}.
\end{proof}

\subsubsection{}
\begin{prop}{propsnow} If $L(\lambda)$ is non-critical and
  snowflake then
  $$Re^{\rho}\ch L(\lambda)=\sum_{w\in W(\lambda)} (-1)^{l(w)} e^{w(\lambda+\rho)}.$$
  \end{prop}
\begin{proof} Since $L(\lambda)$ is snowflake the condition (ii) of
  Theorem \ref{thmsnow} holds. By~\S~\ref{equiv},
$$Re^{\rho}\ch L(\lambda)=\sum_{w\in W(\lambda)} a_w e^{w(\lambda+\rho)}.$$

Since $W(\lambda)$ is  the
Coxeter group corresponding to $\Pi(\lambda)$,
\Lem{lemmaxinorb} implies that $Stab_{W(\lambda)} (\lambda+\rho)=Id$. Therefore
the coefficient of $e^{w(\lambda+\rho)}$ in $Re^{\rho}\ch L(\lambda)$ is equal to $a_w$. In particular, $a_{Id}=1$ and,  by~\Cor{corchN}, $a_w=-a_{r_{\alpha}w}$ if $\alpha\in\Pi(\lambda)$.
The statement is proven.
  \end{proof}

\subsubsection{Remark}
\Prop{propsnow} for the Lie algebra case
is a particular case of~\cite{KT98}; for affine Lie superalgebras
this is Thm. 11.1.2 in~\cite{GK}.

\subsection{}
\begin{cor}{corsnow}
Let  $L(\lambda)$ be a non-critical snowflake module. For each $\beta\in\Pi(\lambda)$
the weight space $M(\lambda)_{r_{\beta}.\lambda}^{\fn^+}$ is one-dimensional
and the maximal submodule of $M(\lambda)$ is generated by
the singular vectors of weights $r_{\beta}.\lambda$ with $\beta \in\Pi(\lambda)$.
\end{cor}
\begin{proof}
Combining~\Thm{thmsnow}  and~\Thm{thmchN} we conclude that
 $\dim M(\lambda)_{r_{\beta}.\lambda}^{\fn^+}=1$ for each $\beta\in\Pi(\lambda)$.
 Let $I$ be the submodule generated by
the singular vectors of weights $r_{\beta}.\lambda$ with $\beta \in\Pi(\lambda)$.
 By~\Thm{thmchN} (ii) $M(\lambda)/I$ is a snowflake module, that is
  $$Re^{\rho}\ch (M(\lambda)/I)=(-1)^{l(w)}
Re^{\rho}\ch (M(\lambda)/I)$$
for every $w\in W(\lambda)$.   By~\S~\ref{equiv},
$$Re^{\rho}\ch (M(\lambda)/I)=
\sum_{w\in W(\lambda)} a_w e^{w(\lambda+\rho)}.$$
Since $\dim (M(\lambda)/I)_{\lambda}=1$ one has $a_{Id}=1$
and thus $a_{w}=(-1)^{l(w)}$. Hence by Proposition \ref{propsnow} $\ch (M(\lambda)/I)=\ch L(\lambda)$, that is
$M(\lambda)/I=L(\lambda)$ as required.
 \end{proof}

\subsection{}
\begin{thm}{thmcomred}
(i) A module $N\in\CO_{\Delta'}$ is snowflake if and only if  every simple
subquotient of $N$
  is also snowflake. In other words
the {full} subcategory of $\CO_{\Delta'}$ consisting of
  snowflake modules is a Serre subcategory.

(ii) Any non-critical snowflake module is completely reducible.
  \end{thm}
  \begin{proof} (i) According to \cite{DGK} $N$ has composition series,
    we denote by $N^{ss}$ the direct sum of all simple
    subquotients of $N$. Note that $N$ and $N^{ss}$ have the same
    characters and therefore $N$ is snowflake if and only if $N^{ss}$
    is snowflake. We therefore may assume that $N$ is semisimple.
Let $\lambda+\rho$ be maximal in $\supp(Re^{\rho}\ch N)$. One has
    $\Delta(\lambda)=\Delta(N)$. Since $N$ is a snowflake
    module, $r_\beta.\lambda<\lambda$ for any
    $\beta\in\Delta^+(\lambda)$. Hence $(\lambda+\rho)(\beta^\vee)>0$.
    By~\Thm{thmsnow}, $L(\lambda)$ is a snowflake module. Now we
    consider $N/L(\lambda)$ and repeat the same argument.
Eventually   we get rid of every simple submodule of $N$.
We have proved that every simple subquotient of $N$ is snowflake.
On the other hand, for the converse just note that if $N_1,N_2\in \CO_{\Delta'}$ then $\Delta(N_1)=\Delta(N_2)=\Delta'$, so
if $N_1,N_2$ are snowflake modules, then $N_1\oplus N_2$ is also
a snowflake module. This establishes (i).

(ii) Let $N$ be an indecomposable snowflake module. Then $N$ is an
    object of $\CO_{\Lambda}$, see~\ref{equiv}. Since $N$ is
    non-critical $\Lambda=W(\lambda).\lambda$ for some $\lambda$. We
    may assume $\lambda$ is maximal in $\Lambda$. By (i) any simple
    subquotient of $N$ is snowflake, hence $N\simeq L(\lambda)$.
\end{proof}

  \subsubsection{}\begin{cor}{snowEnright}
    (i) If $N$ is snowflake, then for any
    $\alpha\in\Pi_{pr}\cap\Delta(N)$ the root space $\fg_{-\alpha}$ acts locally nilpotently
    on $N$.

    (ii) Take  $\alpha\in\Pi_{pr}$. If
    $N,\cC^{\alpha}(N)\in\CO$ and $N$ is snowflake then
$\cC^{\alpha}(N)$ is also snowflake.
  \end{cor}
  \begin{proof}
    By~\Thm{thmsnow}, the assertion (i) holds for simple $N$ and thus for
    $N^{ss}$ in general. The standard $\fsl_2$-argument
    implies that if $\fg_{-\alpha}$ acts locally nilpotently
    on $N^{ss}$, then it acts locally nilpotently on $N$.

    For (ii) consider two cases. If $\alpha\not\in\Delta(N)$, then the
    assertion follows from~\Thm{thmeqO} (iii). If
    $\alpha\in\Delta(N)$, then $\cC^{\alpha}(N)=0$ by (i).
    \end{proof}

  \subsection{}\label{prochee}
For $\lambda\in\fh^*$ denote by $M_{pr}(\lambda)$ (resp.,
$L_{pr}(\lambda)$)
the Verma (resp., simple) $\fg_{pr}$-module
and by $I(\lambda)$ (resp., $I_{pr}(\lambda)$) the maximal proper
subdmodule of $M(\lambda)$ (resp., $M_{pr}(\lambda)$);
 for $w\in W$ we set
$$w\circ \lambda:=w(\lambda+\rho_{pr})-\rho_{pr}.$$

\subsubsection{}
\begin{thm}{corsnowflakes} A $\fg$-module $N$ is snowflake if and only if
$\Res^{\fg}_{\fg_{pr}}N$ is a snowflake $\fg_{pr}$-module.
\end{thm}
\begin{proof}
 Let $N':=\Res^{\fg}_{\fg_{pr}}N$.  By~\Lem{lemOOpr}, $N'\in\CO(\fg_{pr})$.
 {One has $\Delta(N)=\Delta(N')$.}
  Let $\rho_{pr}$ and $R_{pr}$ denote a Weyl vector and the
  Weyl denominator for $\fg_{pr}$ respectively. {Note that}
  $$\frac{R_{pr}e^{\rho_{pr}}}{Re^{\rho}}=e^{\rho_{pr}-\rho}\prod_{\alpha\in\Delta_{\ol{1}}^+}
  (1+e^{-\alpha})\prod_{\alpha\in\Delta_{\ol{0}}^+\setminus\Delta_{pr}}(1-e^{-\alpha})^{-1}$$
  is $W$-invariant, since $\Delta_{\ol{0}}^+\setminus\Delta_{pr}$
  and $e^{\rho_{pr}-\rho}\prod_{\alpha\in\Delta_{\ol{1}}^+}(1+e^{-\alpha})$
  are $W$-invariant. This establishes that if $N$ is  snowflake, then
  $N'$ is snowflake.

Now let $N'$ be snowflake. By~\Thm{thmcomred} we may assume without loss of
generality that $N$ is simple, i.e. $N=L(\nu)$. By~\Thm{thmsnow}, it is enough to verify that
 $(\nu+\rho)(\beta^{\vee})>0$ for each $\beta\in \Pi(\nu)$. Take $\beta\in
  \Pi(\nu)$
  and {choose} $w,\cC$ as in~\Cor{corOnu1} (with $w\beta\in\Pi_{pr}$).
By~\Thm{propid},
  $$\cC(N')=\Res^{\fg}_{\fg_{pr}}\cC(L(\nu))=\Res^{\fg}_{\fg_{pr}} L(w.\nu).$$

  Since $N'$ is a snowflake
  module, $\cC(N')$ is snowflake by~\Cor{snowEnright} (ii). Clearly,
  $L_{pr}(w.\nu)$
  is a subquotient of $\Res^{\fg}_{\fg_{pr}} L(w.\nu)$.
  By~\Thm{thmcomred},  $L_{pr}(w.\nu)$
  is snowflake;  since  $w\beta\in\Pi_{pr}$,~\Cor{snowEnright} (i) gives
$$(w.\nu)((w\beta)^{\vee})\in\mathbb{Z}_{\geq 0}.$$
Since $\rho((w\beta)^{\vee})>0$ we obtain
$(w.\nu+\rho)((w\beta)^{\vee})>0$. Hence
  $$(\nu+\rho)((\beta^{\vee}))=(w.\nu+\rho)((w\beta)^{\vee})>0$$
  as required. This completes the proof.
\end{proof}

\subsubsection{}
\begin{thm}{propAr}
  Let $\lambda\in\fh^*$ be such that $M(\lambda)$ and
  $M_{pr}(\lambda)$ are non-critical modules (over $\fg$ and
  $\fg_{pr}$
  respectively). Let $M'\subset M(\lambda)$
be the $\fg_{pr}$-submodule of $M(\lambda)$ generated by the highest weight vector. (Clearly, $M'$ is isomorphic to  $M_{pr}(\lambda)$.)
If   $L(\lambda)$ is a snowflake module, then
  $I(\lambda)$ is generated by $I_{pr}(\lambda)\subset M'$.
\end{thm}
\begin{proof}
Let $J$ be
  the $\fg$-submodule generated by $I_{pr}(\lambda)$. The module
  $$E:=M'/(M'\cap I(\lambda))$$
  is a $\fg_{pr}$-submodule of
  $L(\lambda)=M(\lambda)/I(\lambda)$, which, by~\Thm{corsnowflakes}, is
a $\fg_{pr}$-snowflake module.
By~\Thm{thmcomred}, $E$ is a snowflake $\fg_{pr}$-module and is
completely reducible (since  $M_{pr}(\lambda)$ is non-critical).
Since $E$ is a quotient of a Verma module, $E$ is simple.
Hence $M'\cap I(\lambda)=I_{pr}(\lambda)$, so
$J\subset I(\lambda)$. {In particular,
$$J\cap M'=I_{pr}(\lambda).$$}
By~\Cor{corsnow} in order to show that $J=I(\lambda)$
it is enough to verify that
\begin{equation}\label{homgJ}
\Hom_{\fg}(M(r_{\beta}.\lambda), J)\not=0
\end{equation}
for each  $\beta\in \Pi(\lambda)$.
Fix $\beta\in \Pi(\lambda)$ and
take $w,\cC$ as in~\Cor{corOnu1} (with $w\beta\in\Pi_{pr}$).
The $\fg$-module
$\cC(M(\lambda))$ contains the $\fg_{pr}$ submodule
$\cC(M')\cong M_{pr}(w\circ \lambda)$; we denote the highest weight vector
of $\cC(M')$ by $v'_{w\circ \lambda}$.
 Note that $M'$ generates
$M(\lambda)$.
Since  both $\cC$ and $\Res_{\fg_{pr}}\cC$  are equivalence of categories
to $\fg_{pr}$, $\cC(M')$ generates
 $\cC(M(\lambda))$. In particular, $\cC(M(\lambda))=U(\fg)v'_{w\circ \lambda}$.

Since $E$ is a snowflake module,~\Cor{corsnow} applied to $\fg_{pr}$
implies that $I_{pr}(\lambda)$
 contains a $\fg_{pr}$-singular vector
 of the weight $r_{\beta}\circ\lambda$. Then,
by~\Prop{thmVerma}, $\cC(I_{pr}(\lambda))$
contains a $\fg_{pr}$-singular vector of the weight $wr_{\beta}\circ\lambda$.
Since $\Res_{\fg_{pr}}\cC$ is an equivalence of categories, it defines an isomorphism of lattices of submodules in $\Res_{\fg_{pr}}M(\lambda)$
  and $\Res_{\fg_{pr}}\cC (M(\lambda))$. In particular, we have
$$\cC(I_{pr}(\lambda))=\cC(J\cap M')=\cC(J)\cap \cC(M').$$

Therefore  by~\Prop{thmVerma}
$\cC(J)\cap \cC(M')$ contains a singular vector of the weight
 $wr_{\beta}\circ\lambda=r_{w\beta}w\circ\lambda$. Since
$w\beta\in\Pi_{pr}$,   the root space
$\fg_{-w\beta}$ acts nilpotently on
the image of $v'_{w\circ \lambda}$ in the quotient
 $$\cC(M')/(\cC(J)\cap \cC(M'))\subset \cC(M(\lambda))/\cC(J).$$
Since $\cC(M(\lambda))=U(\fg)v'_{w\circ\lambda}$,
$\,\fg_{-w\beta}$ acts locally nilpotently
on $\cC(M(\lambda))/\cC(J)$. Since  $\cC(M(\lambda))\cong M(w.\lambda)$,
this means that
$$\Hom_{\fg}(M(r_{w\beta}w.\lambda), \cC(J))\not=0.$$
By applying $\cC^{-1}$ we obtain~(\ref{homgJ}). This completes the proof.
\end{proof}

\subsection{The category $\Snow$}\label{snow}
Recall that $\CO^{inf}$ is a full category of $\fg$-modules $N$ such that
any cyclic submodule of $N$ lies in $\CO$.

\subsubsection{}
\begin{defn}{defnSnow}
Let $\Snow$ be the subcategory of ${\CO}^{inf}$
consisting of non-critical
modules $N$ such that every submodule of $N$ generated by a weight vector is a snowflake module.
\end{defn}

\subsubsection{Remark}
A module in  ${\CO}^{inf}$ does not necessarily admit the character, so we can not use
the formula $w(Re^{\rho}\ch N)=(-1)^{l(w)} Re^{\rho}\ch N$ for defining $\Snow$.

\subsubsection{}
\begin{cor}{snowcat}
The category $\Snow$ is completely reducible.
\end{cor}
\begin{proof}
The standard reasoning (see, for instance,~\cite{GK1}, Lem. 1.3.1)
shows that it is enough to check that any module in $Snow$ has a simple submodule and
$\Ext^1_{\Snow}(L,L')=0$ for simple modules $L,L'\in\Snow$.
Note that  any  cyclic submodule of a module in $\Snow$  is a snowflake module and that
by~\Thm{thmcomred} snowflake modules are completely reducible.
In particular, every module in $\Snow$ has a simple submodule and
for simple modules $L,L'\in\Snow$ one has
$\Ext^1_{{\CO}}(L,L')=0$, so
$\Ext^1_{{\CO}^{inf}}(L,L')=0$.
The claim follows.
\end{proof}

\subsubsection{}
Let $\fg'$ be a symmetrizable Kac--Moody Lie superalgebra with a base $\Sigma'$,
a Cartan algebra $\fh'$ and a Cartan matrix $A'$. Let
$\Sigma\subset\Sigma'$ be a non-empty set consisting of
non-isotropic roots;  denote by $A$
 the corresponding submatrix of $A'$.
The Kac--Moody Lie superalgebra $\fg:=\fg(A)$ can be embedded into $\fg'$,
in such a way that  $\fh:=\fg\cap \fh'$
is the Cartan subalgebra of $\fg$ and $\Sigma$ is the base of $\fg$,
see~\cite{Kbook2}, Ex. 1.2.
Note that $\fg$ is of non-isotropic type.
For $N\in\CO^{inf}(\fg')$ one has $\Res_{\fg'}^{\fg} N\in \CO^{inf}(\fg)$
(this is an advantage of $\CO^{inf}$ comparing to $\CO$).

We  denote by $\Delta$ (resp., $\Delta'$) the
root system of $\fg$ (resp., of $\fg'$) and by
$\Pi_{pr}$ (resp., $\Pi'_{pr}$)  the principal roots for $\fg$
(resp., for $\fg'$). We denote by
$W$ (resp., $W'$)
the Weyl group of $\fg$ (resp., of $\fg'$).
One has $\Delta=\Delta'\cap \mathbb{Z}\Sigma$,
see~\cite{Kbook2}, Ex. 1.2. We denote by $\rho$ (resp., $\rho'$)
the Weyl vector of $\fg$ (resp., $\fg'$).
Observe that for any $\alpha\in W\Sigma$ one has $\alpha^{\vee}\in \fh$
(since $\alpha^{\vee}\in\fh$ for $\alpha\in\Sigma$) and thus
$\rho(\alpha^{\vee})=\rho'(\alpha^{\vee})$.

For a $\fg'$-module $N$
we denote by $\Delta'(N)$ the root subsystem of
$W'\Pi'_{pr}$ introduced in~\S~\ref{Deltalambda};
for a $\fg$-module we denote by $\Delta(M)$
the corresponding root subsystem of $W\Pi_{pr}$. For every
$\lambda\in(\fh')^*$
we denote by $\lambda_{\fh}$ the restriction of
$\lambda$ to $\fh$; we use the notation $\Delta(\lambda)$
 for the root subsystem of $W\Pi_{pr}$  corresponding to $\lambda_{\fh}$ and
we denote by $\Pi(\lambda_{\fh})$ the base of  $\Delta(\lambda_{\fh})$.

\subsubsection{}
\begin{thm}{thmsnowmen}
(i) For $N\in\CO^{inf}(\fg)$ one has
$\Delta(\Res_{\fg}^{\fg'} N)=\Delta(N)\cap \Delta$.

(ii) Take $\lambda\in(\fh')^*$ and set  $M:=\Res_{\fg}^{\fg'} L(\lambda)$. One has
$\Delta(M)=\Delta(\lambda_{\fh})$. Moreover,  $M\in \Snow(\fg)$
if and only if
$(\lambda_{\fh}+\rho)(\beta^{\vee})>0$ for each $\beta\in\Pi(\lambda_{\fh})$.
\end{thm}
\begin{proof}
First, let us show that
\begin{equation}\label{reim}
W'\Pi'_{pr}\cap \Delta=W\Pi_{pr}.
\end{equation}
One has $\Pi_{pr}=\Pi'_{pr}\cap\Delta$. This gives
the inclusion $\supset$. To check the inclusion $\subset$
we  can (and will) assume that $\Sigma$ is connected (and thus
$\Pi_{pr}$ is connected).
Let $\Pi''$ be the connected
component of $\Pi'_{pr}$ which contains $\Pi_{pr}$ and let
$W''$ be the subgroup of $W'$ generated
by $r_{\beta}, \beta\in\Pi''$. Let $\alpha\in\Pi'_{pr}, w\in W'$ be such that
$w\alpha\in \Delta$.
One readily sees that $\alpha\in\Pi''$ and that we may  assume
that $w\in W''$.
Normalizing $(-|-)$ by the condition
$||\alpha||^2=2$, we obtain $||w\alpha||^2=2$ and $||\beta||^2\in\mathbb{Q}_{>0}$
for each $\beta\in\Pi''$, in particular,
for $\beta\in \Pi_{pr}$. By~\S~\ref{symnoniso}, this implies that
{$w\alpha\in\Delta_{re}$. Since $\alpha$ is even,
$w\alpha\in W\Pi_{pr}$; this establishes~(\ref{reim}).}
Now (i) follows from the formula~(\ref{reim}) and the formula
$$\supp(\Res_{\fg}^{\fg'} N)=\{\mu|_{\fh}\ |\ \mu\in\supp(N)\}.
$$

For (ii) note that for $\alpha\in W\Pi_{pr}$ one has $\alpha^{\vee}\in \fh$ so
$\lambda(\alpha^{\vee})=\lambda_{\fh}(\alpha^{\vee})$; this gives
$$\Delta(M)=\{\alpha\in
W\Pi_{pr}|\ \lambda(\alpha^{\vee})\in\mathbb{Z}\}=\{\alpha\in
W\Pi_{pr}|\ \lambda_{\fh}(\alpha^{\vee})\in\mathbb{Z}\}=\Delta(\lambda_{\fh}).$$

From~\Thm{thmsnow} it follows that $M\in \Snow(\fg)$
implies
$(\lambda+\rho)(\beta^{\vee})>0$ for each $\beta\in\Pi(\lambda_{\fh})$.
Now let us assume that $(\lambda+\rho)(\beta^{\vee})>0$ for each $\beta\in\Pi(\lambda_{\fh})$ and
show that
$M\in \Snow(\fg)$, i.e.  that
every simple $\fg$ subquotient $E$ of $M$ is a snowflake $\fg$-module.
Since $E$ is a simple $\fg$-module in the category $\CO(\fg)$, one has
$E=L_{\fg}(\nu)$, where $\nu\in \fh^*$.

Take $\beta\in \Pi(\lambda_{\fh})$.
By~\Lem{lemwmin}, there exists $\alpha_1,\ldots,\alpha_s\in\Pi_{pr}$
such that $\lambda_{\fh}$ is
$(\alpha_1,\ldots,\alpha_s)$-friendly and $w\beta\in\Pi_{pr}$ for
$w:=r_{\alpha_s}\ldots r_{\alpha_1}\in W$.
By (i), $\Delta(\lambda_{\fh})=\Delta' (\lambda)\cap \Delta$, so
 $\lambda$ is
$(\alpha_1,\ldots,\alpha_s)$-friendly.
Let $\cC$ be the Enright functor $\cC^{\alpha_1,\ldots,\alpha_s}$. By~\Thm{propid},
$$\cC(M)=\Res^{\fg'}_{\fg}\cC(L(\lambda))=\Res^{\fg'}_{\fg} L(w(\lambda+\rho')-\rho').$$
One has
$$(w(\lambda+\rho'))(w\beta^{\vee})=(\lambda+\rho')(\beta^{\vee})=(\lambda_{\fh}+\rho)(\beta^{\vee}) $$
and thus $(w(\lambda+\rho'))(w\beta^{\vee})>0$ and
$w\beta\in \Delta'(w(\lambda+\rho')-\rho')$. Since $w\beta\in\Pi_{pr}\subset\Pi'_{pr}$, the root space
 $\fg'_{-w\beta}$
acts locally nilpotently on $\cC(L(\lambda))$ and thus on
$\cC(E)=L_{\fg}(w(\nu+\rho)-\rho)$. Therefore
$(w(\nu+\rho))(w\beta^{\vee})>0$, so $(\nu+\rho)(\beta^{\vee})>0$.
Using~\Thm{thmsnow} we conclude that $E=L_{\fg}(\nu)$ is a  snowflake $\fg$-module.
This completes the proof.
\end{proof}

\subsubsection{}
\begin{cor}{snowcor}
If $(\lambda+\rho)(\beta^{\vee})>0$ for each $\beta\in\Pi(\lambda_{\fh})$, then
$L(\lambda)$
is completely reducible over $\fg$.
\end{cor}

\subsection{Remark}
Non-critical simple snowflake modules appear in~\cite{GK1} as "weakly admissible modules" and some versions of~\Cor{snowcat}, \Prop{snowlevel} are proven there. In Prop. 2.2 of~\cite{GK1}
it is proven that for a snowflake $L(\lambda)$ the space $\Ext^1_{\fg}(L(\lambda),L(\lambda))$
 is naturally isomorphic to $\Delta(\lambda)^\perp$.
 In particular,
 if $L(\lambda)$ is snowflake with $\mathbb{C}\Delta(\lambda)=\mathbb{C}\Delta$,
 then $\Ext^1_{[\fg,\fg]}(L(\lambda),L(\lambda))=0$.
 Similarly to~\Cor{snowcat} we obtain the following corollary.

 \subsubsection{}\begin{cor}{tildecat}
 Let $N$ be a $\fg$-module with the following properties:

 every $v\in N$ generates a finite-dimensional $\fh+\fn^+$-submodule of $N$;

 every simple subquotient of $N$ is a snowflake module $L$
 with  $\mathbb{C}\Delta(L)=\mathbb{C}\Delta$.

 Then $N$ is completely reducible over $[\fg,\fg]$.
 \end{cor}

\subsection{The case of affine $\fg$}\label{snowaffi}
Let $\fg$ be an affine Lie superalgebra of non-isotropic type.
By~\ref{non-isotropictype} the matrix $B$ is of affine type, so $\Delta_{pr}:=\Delta(\fg_{pr})$
is an affine root system.
We normalize the invariant bilinear form
by $(\alpha|\alpha)=2$ for some
 $\alpha\in\Sigma$. Then
$(\alpha|\alpha)\in \mathbb{Q}_{>0}$ for every $\alpha\in\Sigma$ and thus for
every $\alpha\in\Delta_{re}$.

We denote by $\delta$  the minimal imaginary root in $\Delta$ and by
$K\in\fg$ the canonical central element given by $\lambda(K):=(\lambda|\delta)$.
 If $K$ acts on a $\fg$-module $N$
by $k\Id$, then $k$ is called the {\em level} of $N$.
A module of level $k$ is non-critical if and only if $k\not=-h^{\vee}$, where
$h^{\vee}:=(\rho|\delta)\in\mathbb{Q}_{>0}$.
We set
$$\fh^*_k:=\{\lambda\in\fh^*|\ (\lambda|\delta)=k\}.$$
For any root subsystem $\Delta'$ we introduce the set
$$S(k;\Delta'):=
\{\lambda\in\fh^*_k|\ L(\lambda)\in \CO_{\Delta'}\text{ is a snowflake module}\}.$$

For any $\lambda\in\fh^*, x\in\mathbb{C}$ one has
$$\ch L(\lambda+x\delta)=e^{x\delta}\ch L(\lambda)$$
so $S(k;\Delta')$ is invariant under the shifts by $\mathbb{C}\delta$.

\subsubsection{}
\begin{prop}{snowlevel}
(i) If $\Delta(\lambda)$ is finite and non-empty, then
$(\lambda|\delta)\not\in\mathbb{Q}$. If $\Delta(\lambda)$ is infinite, then
$(\lambda|\delta)\in\mathbb{Q}$ and each connected component of
$\Pi(\lambda)$ is of affine type.

(ii) For any $k\in \mathbb{Q}$ there are finitely many
root subsystems $\Delta'$ such that
$\Delta(\lambda)=\Delta'$ for some $\lambda\in\fh^*_k$.

(iii) If $L(\lambda)$ is a snowflake module and $\Delta(\lambda)$ is infinite, then
$(\lambda+\rho|\delta)\in\mathbb{Q}_{>0}$.
In particular, $L(\lambda)$ is non-critical.

(iv) If $\Delta'\subset W\Pi_{pr}$ is such that
$\mathbb{Q}\Delta'=\mathbb{Q}\Delta$, then for any $k$ the set
$S(k;\Delta')/\mathbb{C}\delta$
is finite.

\end{prop}

\begin{proof}
Since $\Delta_{pr}$ is an affine root system, there exists $r\in\mathbb{Z}_{>0}$ such that
$\Delta_{pr}+r\delta=\Delta_{pr}$.
For $s\in\mathbb{Z}$ and $\alpha\in {\Delta}_{re}$
one has
$(sr\delta+\alpha)^{\vee}=\alpha^{\vee}+\frac{2sr}{||\alpha||^2}\delta$
so for $\alpha\in\Delta(\lambda)$ one has
\begin{equation}\label{eq8}
\Delta(\lambda)\cap \{\alpha+\mathbb{Z}r\delta\}=\{sr\delta+\alpha|\ \
\frac{2sr}{||\alpha||^2}(\lambda|\delta)\in\mathbb{Z}\}.\end{equation}

If $(\lambda|\delta)\not\in\mathbb{Q}$, this set is equal to $\{\alpha\}$;
otherwise this set is infinite.
This gives (i).

For (ii) fix  $k\in\mathbb{Q}$. By~(\ref{eq8}) there exists $q\in\mathbb{Z}_{>0}$
  such that for any $\lambda\in\fh^*_k$ one has
  $\alpha\in\Delta(\lambda)$ if and only if $\alpha+q\delta\in\Delta(\lambda)$.
Let $\lambda\in\fh^*_k$ be such that $\Delta(\lambda)$ is non-empty.
Let us show that for each $\alpha\in\Pi(\lambda)$ one has $\alpha<q\delta$.
Indeed, take $\alpha\in\Pi(\lambda)$. There exists $j\geq 0$ such that for
$\alpha':=\alpha-jq\delta$ one has $0<\alpha'<q\delta$; by above
$\alpha'\in\Delta(\lambda)^+$. If $\alpha'\not=\alpha$ then
 $r_{\alpha}\alpha'\in\Delta^+$ (since  $\alpha\in\Pi(\lambda)$);
 this gives $-\alpha-jq\delta\in\Delta^+$,
a contradiction. Thus $\alpha<q\delta$.
Since   $\{\alpha\in\Delta|\ 0<\alpha<q\delta\}$ is finite, this establishes (ii).

For (iii), (iv) we will use the following observation:
if $L(\lambda)$ is a snowflake module, then $(\lambda+\rho|\alpha)\in\mathbb{Q}_{>0}$
for each $\alpha\in\Pi(\lambda)$ (this follows from~\Thm{thmsnow}
 and the fact that $(\alpha|\alpha)\in \mathbb{Q}_{>0}$
for any $\alpha\in \Delta_{re}$).

For (iii) assume that $\Delta(\lambda)$ is infinite. Then
$s\delta\in\mathbb{Z}_{\geq 0}\Delta(\lambda)$ for some $s\in\mathbb{Z}_{>0}$,
that is $s\delta\in\mathbb{Z}_{\geq 0}\Pi(\lambda)$ (see~\Lem{lemPilambda}
(iii)). The above observation implies
$(\lambda+\rho|s\delta)>0$ and (iii) follows.

For (iv) fix $k\in\mathbb{Q}$ and take $q>0$ as above.
Fix $\Delta'$ and set  $\Pi':=\Pi(\Delta')$.
Since $\mathbb{Q}\Pi'=\mathbb{Q}\Pi$ one has
$\{\mu|\ \forall \alpha\in\Pi'\ \ (\mu|\alpha)=0\}=\mathbb{C}\delta$,
 it is enough to show that the set
$$H:=\{\lambda(\alpha^{\vee})|\ \lambda\in S(k;\Delta'),\alpha\in\Pi'\}$$
is finite. By definition of $\Delta(\lambda)$ one has $H\subset \mathbb{Z}$, so it is enough to verify that $H$ is a bounded set.

Take $\lambda\in S(k;\Delta')$ and $\alpha\in\Pi'$. By above,
$(\lambda+\rho|\alpha)>0$,
so $H$ is bounded from below.
By (i) each  connected component of $\Pi'$ is of affine type, so
 by~\ref{FINAFFIND},
$$q\delta=\sum_{\beta\in\Pi'} m_{\beta}\beta,\ \ \ m_{\beta}>0.$$
and thus, using the above observation, we obtain
$$q(k+h^{\vee})=(\lambda+\rho|q\delta)=\sum_{\beta\in\Pi'} m_{\beta}
(\lambda+\rho|\beta)>m_{\alpha}(\lambda+\rho|\alpha).$$
Since $m_{\alpha}>0$, the set $H$ is bounded. This gives (iv).
\end{proof}

\subsubsection{Remark}
It is not hard to classify the possible levels of snowflake modules for a given root subsystem $\Delta'$.
The root subsystems satisfying the assumption $\mathbb{Q}\Delta'=\mathbb{Q}\Delta$ are the most interesting;
for these root subsystems in the case when $\fg$ is an affine Lie algebra
the sets $S(k;\Delta')$ are described in~\cite{KW0}.

\subsubsection{}
\begin{lem}{lemiv}
Let $y\in GL(\fh^*)$ preserves $\Delta$ and the form $(-|-)$ (i.e. $y\Delta=\Delta$ and
$(y\lambda|y\nu)=(\lambda|\nu)$ for $\lambda,\nu\in\fh^*$).
If a root subsystem $\Delta'$ is such that
  $y (\Pi(\Delta'))\subset\Delta^+$, then
$S(k;y\Delta')=y.S(k;\Delta')$  (where $y.\nu:=y(\nu+\rho)-\rho$).
\end{lem}
\begin{proof}
Take $\alpha\in\Delta_{re}$. Recall that
 $\rho(\alpha^{\vee})\in\mathbb{Z}$ (resp., $\rho(\alpha^{\vee})\in\mathbb{Z}+\frac{1}{2}$)
if $\frac{\alpha}{2}\not\in\Delta$ (resp.,  if $\frac{\alpha}{2}\in\Delta$).
Since $y$  preserves $\Delta$ and the form $(-|-)$ one has
 $(\rho-y\rho)(\alpha^{\vee})\in\mathbb{Z}$.  Thus for any $\lambda\in\fh^*$ one has
$$\Delta(y.\lambda)=\Delta(y\lambda+y\rho-\rho)=\Delta(y\lambda)=y\Delta(\lambda).$$
Using~\Lem{lemPilambda} (ii), we see that
the condition
$y\Pi(\Delta')\subset\Delta^+$ gives $\Pi(y\Delta')=y\Pi(\Delta')$.
Since $(y.\lambda+\rho)(y\alpha^{\vee})=
(\lambda+\rho)(\alpha^{\vee})$, this
establishes the assertion.
\end{proof}

\section{Arakawa's Theorem for $\mathfrak{osp}(1|2\ell)^{(1)}$}
In this section we deduce Arakawa's Theorem for
$\mathfrak{osp}(1|2\ell)^{(1)}$
from Arakawa's Theorem for the Lie algebra $\mathfrak{sp}_{2\ell}^{(1)}$.
In~\cite{A} T.~Arakawa proved this theorem  for all non-twisted  affine Lie algebras.
Among non-twisted affine superalgebras (which are not Lie algebras)
only $\mathfrak{osp}(1|2\ell)^{(1)}$  have non-isotropic type.

\subsection{Admissible levels}\label{admil}
Let
$\fg:=\dot{\fg}^{(1)}$  be
a non-twisted  affine Lie superalgebra of non-isotropic type.
Retain notations of~\ref{snowaffi}.
We  denote by  $\dot{\Sigma}$  the base of $\dot{\fg}$ and by
 $\Sigma=\dot{\Sigma}\cup\{\alpha_0\}$  the base of $\fg$.
Let $\Lambda_0\in\fh^*$ be the $0$th fundamental root, i.e.
  $(\Lambda_0|\alpha_0^{\vee})=1$ and
  $(\Lambda_0|\dot{\Sigma})=0$.

The definition of admissible weights in~\cite{KWrat} can be written as follows.

\subsubsection{}\begin{defn}{}
A weight $\lambda\in\fh^*$ is called {\em admissible}
if $\mathbb{Q}\Delta(\lambda)=\mathbb{Q}\Delta$ and

$(\lambda+\rho)(\alpha^{\vee})>0$ for each
$\alpha\in \Pi(\lambda)$.
\end{defn}

In the light of~\Thm{thmsnow}, we
can reformulate this definition as follows: $\lambda$ is admissible
if $\mathbb{Q}\Delta(\lambda)=\mathbb{Q}\Delta$  and $L(\lambda)$ is a snowflake module. Note that  the admissible weights  are non-critical (see~\Prop{snowlevel} (iii)).

\subsubsection{}\begin{defn}{}
The level $k$ is called {\em admissible} if
$k\Lambda_0$ is admissible.\end{defn}

If $k$ is admissible, then
$\Pi(k\Lambda_0)=\dot{\Sigma}\cup\{\alpha_0'\}$, where $\alpha_0'$ is some root not in $\Delta(\dot\fg)$. In fact,  $\alpha_0'=\alpha_0+j\delta$ or $\alpha_0'=-\theta_s+(j+1)\delta$, where $j\geq 0$ and $\theta_s$ is the
maximal short root in ${\Delta}(\dot{\fg})$.

\subsubsection{}\label{gg}
A $[\fg,\fg]$-module  (resp., $\fg$-module) $N$ is called
{\em restricted} if
for every $a\in\dot{\fg}, v\in N$ there exists $n$ such that $(at^m)v=0$
for every $m>n$. The modules in $\CO$ are
 restricted $\fg$-modules.

Let $N$ be a non-critical restricted $[\fg,\fg]$-module of level $k$. The Sugawara construction
equips $N$ with an action of the Virasoro algebra
$\{L_n\}_{n\in\mathbb{Z}}$, see~\cite{Kbook2},
12.8 for details.
Moreover, the $[\fg,\fg]$-module structure on $N$ can be extended to a
$\fg$-module
structure by setting $d|_N:=-L_0|_N$.

For a restricted $\fg$-module the action of $L_0$ and the Casimir element $\Omega$
are related by the formula $\Omega=2(K+h^{\vee})(d+L_0)$. Therefore,
the above procedure assigns to a restricted non-citical $[\fg,\fg]$-module of
level $k$ a
restricted $\fg$-module of level $k$ with the zero action of the Casimir operator.
We denote by $\CO_k$ the full subcategory of of $\CO(\fg)$
consisting of  the modules of level $k$ with the zero action of  the Casimir operator.
{If $L(\lambda)$ is of level $k\not=-h^{\vee}$, then
there exists a unique $x\in\mathbb{C}$ such that  $L(\lambda+s\delta)\in\CO_k$.}

\subsubsection{}
We denote by $Vac^k$ the vacuum module:
$$Vac^k:=M(k\Lambda_0)/\sum_{\alpha\in\dot{\Sigma}} M(r_{\alpha}.k\Lambda_0).$$
The vacuum module has a unique simple quotient; it is isomorphic to $L(k\Lambda_0)$.
Similarly to~\cite{FZ}, the
vacuum module $Vac^k(\fg)$ has a structure of a vertex superalgebra with
\begin{equation}\label{Yxt}
Y(at^{-1}\vac,z)=\sum_{n\in \mathbb{Z}} (at^n)z^{-n-1}\ \text{ for }a\in\dot{\fg};
\end{equation}
we denote this vertex superalgebra by $V^k(\fg)$. One has the natural correspondence
between the ``weak'' $V^k(\fg)$-modules and
 the restricted $[\fg,\fg]$-module of level $k$. For $k\not=-h^{\vee}$
 the category $\CO$ for $V^k(\fg)$ is naturally equivalent  to the category $\CO_k$.
The maximal proper submodule $I(k)$ of  $Vac^k$
is the maximal ideal in the vertex algebra $V^k(\fg)$.
Moreover, if $I(k)$ is generated by a set $E$ as a $\fg$-module, then
$I(k)$ is generated by $E$ as an ideal in $V^k(\fg)$. In particular,
$L(k\Lambda_0)$ inherits a structure of
a vertex superalgebra, which is simple; it is denoted by
$V_k(\fg)$.

\subsubsection{Remark}
By~\Prop{snowlevel} (i),
 $L(k\Lambda_0)$ is snowflake if and only if
$k\not\in\mathbb{Q}$ or $k$ is admissible.
If $k\not\in\mathbb{Q}$, then $Vac^k$ is simple (see~\cite{GKvac})
and so $V^k(\fg)=V_k(\fg)$.

\subsubsection{}\label{ThmAr}
Let $\ft$ be a Lie algebra.
If $L(k\Lambda_0)$ is an integrable module (i.e., $k\in\mathbb{Z}_{\geq 0}$), then the $V_k(\ft)$-modules
correspond to the restricted integrable $\ft$-modules of level $k$,
see~\cite{DLM}, Thm. 3.7; these
modules are completely reducible and the irreducible modules are
the integrable highest weight modules of level $k$ (there are finitely many
such modules and $V_k(\ft)$ is a {\em rational} vertex algebra).
The following Adamovi\'c-Milas conjecture~\cite{AM}  was proven by T.~Arakawa in~\cite{A}.

{\em Theorem (Arakawa, 2014).
Let $k$ be an admissible level for an affine Lie algebra $\ft$.
The $V_k(\ft)$-modules in category $\CO$ are completely reducible and the irreducible modules correspond to $L_{\ft}(\lambda)\in\CO_k$,
 where $\lambda$ belongs to the set of the  admissible weights of level $k$
 with $\Delta(\lambda)\cong \Delta(k\Lambda_0)$; this set is finite}.

\subsubsection{Remark}\label{isomofDeltas}
The formula $\Delta(\lambda)\cong \Delta(k\Lambda_0)$ means that
this root systems are isomorphic as abstract root systems;
by~\cite{KW0} 
 $\Delta(\lambda)\cong \Delta(k\Lambda_0)$
if and only if $\Delta(\lambda)=y\Delta(k\Lambda_0)$, where
$y$ is an element of the extended Weyl group of $\Delta$.

 \subsection{Remark}\label{remara}
We claim that Arakawa's Theorem can be reformulated as follows:

 \begin{thm}{}  If $k$ is admissible,
then the $\CO$-category for $V_k(\ft)$-modules
is a finite direct sum
\begin{equation}\label{snowara}
\bigoplus_{\Delta':\ \Delta'\cong \Delta(k\Lambda_0)} (Snow\cap \CO_k\cap \CO_{\Delta'}),
\end{equation}
where $Snow\cap \CO_k\cap \CO_{\Delta'}$ is the category of snowflake modules in
$\CO_{\Delta'}$, which have level $k$ and are annihilated by the Casimir.
\end{thm}

 Indeed, by~\Prop{snowlevel} (ii) the sum in~(\ref{snowara}) is finite.
 If $\Delta'\cong \Delta(k\Lambda_0)$, then
 the category
$Snow\cap \CO_k\cap \CO_{\Delta'}$
is semisimple (by~\Thm{thmcomred}) and  the number of equivalence classes of simple objects  in this category is finite
(see~\ref{gg} and~\Prop{snowlevel} (iv)).

\subsubsection{Remark}
Let $\fs=\dot{\fs}^{(1)}$ be a non-twisted affine Lie superalgebra (of
isotropic or non-isotropic type). Then $\dot{\fs}$ can have one, two or three simple components.
 Let $\dot{\fs}^{\#}$ be the ``largest'' simple component of the reductive Lie algebra $\dot{\fs}_{\ol{0}}$ \footnote{By the "largest" we mean that
    a simple Lie algebra
of  rank $n$ is ``larger'' than  a simple Lie algebra  of rank $m$ if
$n>m$ and that $C_n$ is the largest among the non-exception
simple Lie algebras of rank $n$}. View
 ${\fs}^{\#}:=(\dot{\fs}^{\#})^{(1)}$ as
 a subalgebra of  $\fs_{\ol{0}}$. Thm. 5.3.1 in~\cite{GS3} states that

{\em if $L(k\Lambda_0)$ is integrable over  ${\fs}^{\#}$, then the $V_k(\fs)$-modules
correspond to the restricted integrable $\fs$-modules of level $k$,
 which are integrable over ${\fs}^{\#}$}.

\subsection{}
\begin{thm}{thmAr12l}
Let $\dot{\fg}=\mathfrak{osp}(1|2\ell)$ and let $k$ be an admissible level.
The $V_k(\fg)$-modules in the category $\CO$    are completely reducible
and the irreducible modules correspond to $L(\lambda)$, where
$\lambda$ is admissible and $\Delta(\lambda)\cong \Delta(k\Lambda_0)$;
up to isomorphisms,
there are finitely many simple $V_k(\fg)$-modules in the category $\CO$.
\end{thm}

\begin{proof}
Our  proof
 is based on the Arakawa's Theorem for $\dot{\fg}_{\ol{0}}=\mathfrak{sp}_{2\ell}$
and~\Thm{propAr}.

One has $\fg_{pr}=\fg_{\ol{0}}=\mathfrak{sp}_{2\ell}^{(1)}$.
The module $Vac^k$ contains a natural copy
of $\fg_{pr}$-vacuum module  $Vac_{pr}^k$ (which is the $\fg_{pr}$-submodule
generated by the highest weight vector in $Vac^k$).

Let $I$  be the maximal proper submodule of $Vac^k$. A restricted $[\fg,\fg]$-module
$N$ correspond to
a $V_k(\fg)$-module if and only if $N$ is annihilated by all fields $a_{(m)}$
with $a\in I$, $m\in\mathbb{Z}$. Since
 $$Vac^k=M(k\Lambda_0)/ \sum_{\alpha\in\dot{\Sigma}}M(r_{\alpha}.k\Lambda_0),$$
\Thm{propAr} implies that $I$ is generated by $I':=I\cap Vac_{pr}^k$.

Note that the vertex algebra $V^k(\fg_{\ol{0}})$ is a subalgebra
of $V^k(\fg)$. The standard reasoning (see, for example,~\cite{AM}, Prop. 3.4)
shows that a restricted $[\fg,\fg]$-module
$N$ correspond to
a $V_k(\fg)$-module if and only if $N$ is annihilated by all fields $a_{(m)}$
with $a\in I'$, $m\in\mathbb{Z}$.
We conclude that a restricted $[\fg,\fg]$-module $N$ of level $k$
is a $V_k(\fg)$-module if and only if $Res_{\fg_{\ol{0}}}^{\fg} N$
is a $V_k(\fg_{\ol{0}})$-module. Thus the $V_k(\fg)$-modules in the
category $\CO$ correspond to  $N\in\CO_k$ such that
$Res_{\fg_{\ol{0}}}^{\fg} N$
is a  $V_k(\fg_{\ol{0}})$-module.

Since $k$ is an admissible level for $\fg$,
 $L(k\Lambda_0)$ is a snowflake module and thus
  $L_{\fg_{\ol{0}}}(k\Lambda_0)$ is a snowflake  $\fg_{\ol{0}}$-module
 (by~\Thm{corsnowflakes}). One has
$$\mathbb{Q}\Delta(k\Lambda_0)=\mathbb{Q}\Delta=\mathbb{Q}\Delta_{\ol{0}},$$
so $k$ is an admissible level for $\fg_{\ol{0}}$.

Take $N\in\CO_k$ and set $N':=Res_{\fg_{\ol{0}}}^{\fg} N$. Recall that
$\Delta(N')=\Delta(N)$.
By above, $N$ is  a  $V_k(\fg)$-module
if and only if $N'$ is a  $V_k(\fg_{\ol{0}})$-module. The latter,
 by Arakawa's Theorem,  is equivalent to the condition that $N'$
 is   a snowflake module of level $k$ with
$\Delta(N')\cong \Delta(k\Lambda_0)$.
By~\Thm{corsnowflakes}, $N'$ is a snowflake module if and only if $N$ is a snowflake module.
We conclude that  $N$ is  a  $V_k(\fg)$-module if and only if
$N$ is a snowflake module of level $k$ with $\Delta(N)\cong \Delta(k\Lambda_0)$,
which gives~(\ref{snowara}). Now the statement follows from~\ref{remara}.
\end{proof}

\subsection{}
Let $\tilde{\CO}^{inf}_k$ be the full category of $\fg$-modules $N$ of level $k$ such that
every $v\in N$ generates a finite-dimensional $\fh+\fn^+$-submodule. (In other words, we replace the condition that $\fh$ acts diagonally by
  the condition  that $\fh$ acts locally finitely.)
Clearly, all simple subquotients of a module $N\in \tilde{\CO}^{inf}_k$
lie in $\CO_k$. Let $\tilde{\CO}^{inf}$ be
the corresponding categories of $V^k(\fg)$-modules and $V_k(\fg)$-modules.

\subsubsection{}
\begin{cor}{Arakawatilde}
Let $\dot{\fg}$ be a simple Lie algebra or $\mathfrak{osp}(1|2\ell)$ and let $k$ be an admissible level.
The $V_k(\fg)$-category $\tilde{\CO}^{inf}$    is semisimple.
\end{cor}
\begin{proof}
Take $N\in \tilde{\CO}^{inf}$ and view $N$ as a $\fg$-module.
Let $L$ be a simple $\fg$-subquotient of $N$. Then $L$ is a simple
$V^k(\fg)$-module which is annihilated by all fields $a_{(m)}$ with $a\in I$,
so $L\in\CO$ is a simple $V_k(\fg)$-module. Combining~\Thm{thmAr12l} and~\Cor{tildecat}
we conclude that $N$ is completely reducible.
\end{proof}

\subsection{Remarks}\label{occasional}
One can give another natural definition of the integral root system
by including odd non-isotropic roots
$$\Delta_{super}(\lambda):=\{\alpha\in\Delta_{re}|\ \lambda(\alpha^{\vee})\text{ is integral and is even if $\alpha$ is odd}\},$$
i.e., $\Delta_{super}(\lambda)_{\ol{0}}=\Delta(\lambda)$ and
$\Delta_{super}(\lambda)_{\ol{1}}=\{\alpha\in\Delta|\ 2\alpha\in\Delta(\lambda)\}$.
It turns out that the Arakawa Theorem does not hold if we substitute $\Delta(\lambda)$ by
$\Delta_{super}(\lambda)$,  see~\ref{ell1}.

\subsubsection{}\label{ell1}
Consider the case $\mathfrak{osp}(1|2)$.
The admissible weights for $\mathfrak{osp}(1|2)$ are described in~\cite{KRW}, Sect. 6.
Let $\lambda\in\fh^*_k$ be an admissible  weight. Then $k$ is an admissible level
and $\Delta(\lambda)\cong\Delta(k\Lambda_0)$ is  of the type $A_1^{(1)}$ (i.e., the Dynkin diagram of
$\Pi(\lambda)$ is of the type $A_1^{(1)}$). Therefore $L(\lambda)$ is a $V_k(\fg)$-module.
The root system
$\Delta_{super}(k\Lambda_0)$ is of type
$B(0|1)^{(1)}$ and the root system $\Delta_{super}(\lambda)$
can be $B(0|1)^{(1)}, A_1^{(1)}, C(2)^{(2)}$. Hence $\Delta_{super}(\lambda)$
is not always isomorphic to $\Delta_{super}(k\Lambda_0)$.

\subsubsection{}
For $\ell>1$ there are admissible weights which have non-admissible levels
(i.e., $\lambda\in\fh^*_k$ is admissible and $k\Lambda_0$ is not admissible) and
there are admissible levels $k$ such that $\Delta(\lambda)\not\cong\Delta(k\Lambda_0)$ for some admissible weights $\lambda\in\fh^*_k$.

\section{Snowflake modules for isotropic type}
When we try to define snowflake modules for Kac--Moody superalgebras of isotropic
type we encounter several problems.

The first problem concerns the definition itself since
$Re^{\rho}$ is not skew-invariant with respect to the Weyl group action introduced in~\ref{Weyldenom}. This means that if we literally extend the definition of snowflake
modules to this case,  then  the trivial module
would not to satisfy this definition. (However, $L(\lambda)$ would satisfy
such definition if $\lambda+\rho$ is a typical weight and
 $(\lambda+\rho)(\alpha^{\vee})>0$ for each $\alpha\in\Pi(\lambda)$,
see~\Prop{cortyp}). This can be remedied by introducing another (in fact, more natural)
action of the Weyl group as
we discuss below in~\ref{disc}.
Another way to deal with this issue is to consider the restriction to
$\fg_{pr}$.
If $\fg$ is finite-dimensional these two
approaches are equivalent.

The second problem appears in the case when $\Delta(N)$ is infinite.
Roughly speaking, snowflake modules exist only for connected $\Pi(N)$.
For example, the infinite-dimensional snowflake modules $N$ with
$\Delta(N)=\Delta_{\ol{0}}$
exist only for twisted affinizations of $\mathfrak{sl}(1|n)$ and
$\mathfrak{osp}(2|2n)$. It makes sense to consider invariance
condition with respect to a suitable root
subsystem of $\Delta(N)$ in the spirit of~\cite{KW}.

Finally, Theorem \ref{thmsnow} does not hold for the isotropic type, although
it holds for ``weakly typical'' weights
$\lambda$ such that
$(\lambda+\rho|\alpha)\neq 0$ for certain
finite set of isotropic $\alpha$.

In this section we assume that $\fg$ has isotropic type.

Retain notations of~\ref{sigmapr}.
From the classification theorem~\cite{H} (see~\S~\ref{isotropictype})
it follows
that  the equality
$$
W\Pi_{pr}=\cup_{\Sigma\in\cB}W_\Sigma\Sigma_{pr}
$$
 holds only for finite-dimensional $\fg$.
 If $\fg$ is infinite-dimensional all $W_\Sigma$ are finite, moreover,
 the set $\cup_{\Sigma\in\cB}W_\Sigma\Sigma_{pr}$ is
 finite while $ W\Pi_{pr}$ is infinite.

\subsection{Snowflake modules for $\dim\fg<\infty$}
Now let $\fg$ be finite-dimensional. In this case
$\fg_{pr}=\fg_{\ol{0}}$ and { the Weyl vector}
$\rho_{pr}=\rho_{\ol{0}}$ is well-defined;   one has
$$\rho_{\ol{0}}=\rho+\frac{1}{2}\sum_{\alpha\in\Delta_{\ol{1}}}\alpha.$$
{Define $\Sigma(\lambda)$ as in~\ref{sigmapr}.}
\subsubsection{}
\begin{lem}{lemuseful}
 For any $\lambda\in\fh^*$, any base
$\Sigma\in\cB$  and every $\beta\in \Sigma(\lambda)$
one has
\begin{equation}\label{usual0}
r_{\beta}\bigl(R_{\ol{0}}e^{\rho_0}\ch L(\lambda)\bigr)=-
R_{\ol{0}}e^{\rho_0}\ch L(\lambda)\ \Longleftrightarrow\
(\lambda+\rho)(\beta^{\vee})>0.\end{equation}
\end{lem}
\begin{proof}
The implication  $\Longleftarrow $ follows from $W$-invariance of
$R_{\ol{1}}e^{\rho_{\ol{0}}-\rho}$
and~\Cor{corchN}. Since $R_{\ol{1}}e^{\rho_{\ol{0}}-\rho}$ is
$W$-invariant by the same argument as in
the proof of~\Thm{thmchN}, we can use Enright functors to reduce
$\Longrightarrow$ to the case $\beta\in  \Sigma_{pr}$.

Assume that $(\lambda+\rho)(\beta^{\vee})\leq 0$ for $\beta\in\Sigma_{pr}$.
Define the linear map $P$ by
$$P(\sum_{\nu\in \fh^*} a_{\nu}e^{\nu}):=\sum_{j\in \mathbb{Q}} a_{j\beta}e^{j\beta}.$$
By~(\ref{suppN}) one has $P(e^{-\lambda}R\ch L(\lambda))=1$.
Since $\supp(e^{-\lambda}R\ch L(\lambda))$ and $\supp R_{\ol{1}}$ lie
in $-Q^+$, the condition $\beta\in\Sigma_{pr}$ gives
$$P \bigl(e^{-\lambda}R_{\ol{0}}\ch L(\lambda)\bigr)=P(R_{\ol{1}})=\left\{
\begin{array}{ll}
1&\text{ if }\beta/2\not\in\Delta,\\
1+e^{-\beta/2}&\text{ if }\beta/2\in\Delta.\end{array}
\right.$$
and thus  the left-hand side of~(\ref{usual0}) does not hold.
This establishes~(\ref{usual0}).
\end{proof}

\subsubsection{}
As follows from classification of finite-dimensional Kac--Moody
superalgebras, we have the following two cases:
\begin{enumerate}
\item if $\fg$ is type I then $\Pi_{pr}\subset \Sigma$ for some $\Sigma\in\cB$;
  \item if $\fg$ is type II then $\Pi_{pr}=\Pi_1\coprod \Pi_2$  and there  exist $\Sigma_1,\Sigma_2\in\cB$
such that $\Pi_i\subset \Sigma_{i,pr}$ for $i=1,2$.
  \end{enumerate}
In the latter case we
denote by $\rho_1,\rho_2$ the Weyl vectors for $\Sigma_1$ and
$\Sigma_2$ respectively.
We set $W_i:=W(\Pi_i)$
(so $W=W_1\times W_2$).
In the former case we  assume $\Sigma_1:=\Sigma_2:=\Sigma$ and
$\Pi_1:=\Pi_2:=\Pi_{pr}$.

\subsubsection{}
\begin{thm}{}
  Let $\dim\fg<\infty$ and $L$ be a simple highest weight
  $\fg$-module.
  Let $\lambda_i$ be the highest weight of $L$ with respect to $\Sigma_i$.
Then
 $\Res^{\fg}_{\fg_{\ol{0}}} L$ is a snowflake module if and only if
$(\lambda_i+\rho_i)(\beta^{\vee})>0$ for every
$\beta\in\Pi(\lambda_i)\cap W_i\Pi_i$ and $i=1$ (resp., $i=1,2$) for
type I (resp., type II).
 \end{thm}
\begin{proof}
The group $W(L)$
is generated by $r_{\alpha}$ with $\alpha\in\Pi(L)$.
For $\fg$ of type I the statement follows from~(\ref{usual0}). Now let
$\fg$ be of type II.
Since  ${W}=W(\Pi_1)\times W(\Pi_2)$ one has
$${W}\Pi_{pr}=W_1\Pi_1\coprod W_2\Pi_2.$$
If $\beta\in W_i\Pi_i$ lies in $\Delta(L)$, then
$\beta\in \Delta_{\Sigma_i}(\lambda_i)$ (since $\Pi_i\subset\Sigma_{i,pr}$). Therefore
$$\Delta(L)=\coprod_{i=1}^2 (\Delta_{\Sigma_i}(\lambda_i)\cap W_i\Pi_i),\ \ \
\Pi(L)=\coprod_{i=1}^2 (\Sigma_i(\lambda_i)\cap W_i\Pi_i).$$
and the assertion  follows from~(\ref{usual0}).
\end{proof}

\subsection{Remarks}
\subsubsection{}
If $\Res^{\fg}_{\fg_{\ol{0}}} L(\lambda)$ is a snowflake module, it is
completely reducible over $\fg_{\ol{0}}$ by~\Thm{thmcomred}.
Sometimes we can have complete reducibility with respect to another
subalgebra of
$\fg$ as one can see from the following
example.

Let $\Res^{\fg}_{\fg_{\ol{0}}} L(\lambda)$ be a snowflake module
for $\fg:=\mathfrak{osp}(2m+1|2n)$. Then $L(\lambda)$  is completely reducible over
$\fg_{\ol{0}}=\mathfrak{o}(2m+1)\times \mathfrak{sp}(2n)$. Note that $\fg$ contains a copy of
$\mathfrak{osp}(1|2n)$; using~\Thm{corsnowflakes} we see that
 $L(\lambda)$ is is completely reducible over
 $\mathfrak{osp}(1|2n)$ (note that $\fg$ does not contain
$\mathfrak{o}(2m+1)\times \mathfrak{osp}(1|2n)$).

\subsubsection{}\label{disc}
Consider the example  $\fsl(2|1)$ with $\Sigma=\{\alpha_1,\alpha_2\}$
with isotropic $\alpha_1,\alpha_2$. We have $\rho=0$ and
 $W_{\Sigma}=\{Id\}$, ${W}=\{Id,r\}$, where $r:=r_{\alpha_1+\alpha_2}$.

Set $x_i:=e^{-\alpha_i}$, we have
$$Re^{\rho}=\frac{1-x_1x_2}{(1+x_1)(1+x_2)}=1+\sum_{j=1}^{\infty}(-1)^j
(x^{j}_1+x^{j}_2).
$$
One has $r(x_1)=x_2^{-1}$ and
$$r(Re^{\rho})=1+\sum_{j=1}^{\infty}(-1)^j (x^{-j}_1+x^{-j}_2)\not=-Re^{\rho}.$$
This shows that if we literally extend the definition of snowflake
modules to $\fsl(2|1)$,  then  the trivial module
would not to satisfy this definition.

If we consider the natural action of $W$ on the ring of rational
functions in $x_1,x_2$, then $r(Re^{\rho})=-(Re^{\rho})$. Note also that
a rational function $Re^{\rho}$ does not depend on the choice of
$\Sigma'\in\cB$ (i.e., $R'e^{\rho'}=\frac{1-x_1x_2}{(1+x_1)(1+x_2)}$
for any $\Sigma'\in\cB$).

\subsubsection{}
The above approach was generalized for symmetrizable affine Lie superalgebras in~\cite{GK}, 2.2.
Using this new action of ${W}$ we can define snowflake modules as modules satisfying the condition
$r_{\alpha}(\ch N)=\ch N$
for $\alpha\in\Pi(N)$. According to this definition $N$ is snowflake if and only if
$\Res^{\fg}_{\fg_{pr}} N$ is a snowflake $\fg_{pr}$-module;
  in particular, $N$ is snowflake if $\Res^{\fg}_{\fg_{pr}} N$ is integrable.
As before, $L(\lambda)$ is a snowflake module if
$\lambda+\rho$ is a  typical weight and
 $(\lambda+\rho)(\alpha^{\vee})>0$ for each $\alpha\in\Pi(\lambda)$.

\section{Index of Notations}
\ref{contra} \hspace{1cm} $\fg(A),\ \Sigma,\ \cB,\ \Delta,\ \Delta_{\ol{0}},\  \Delta_{\ol{1}},\  r_{\alpha}\Sigma$;

\ref{fgB} \hspace{1cm} $\fg_{pr}$, matrix $B$, principal roots; $\Pi_{pr}$;
 \hspace{0.7cm}
\ref{parord}  \hspace{0.2cm} partial order $\leq $ on $\fh^*$;

\ref{Weylgroup}  \hspace{1cm} $W$, $\Delta_{re}$, $\Delta_{im}$, $\ol{\Delta}_{re}$;
 \hspace{0.7cm}
\ref{wbetavee}  \hspace{0.3cm} $\alpha^{\vee}$;

\ref{types} \hspace{1cm} non-isotropic type; isotropic type;
 \hspace{0.7cm}
\ref{FINAFFIND} \hspace{0.2cm} (FIN), (AFF), (IND);

\ref{notat1} \hspace{1cm} $\Delta^+$, $Q^+$;
\hspace{0.7cm}
\ref{useful} \hspace{0.3cm} $\rho$, $w.$;
\hspace{0.7cm}
\ref{Weyldenom} \hspace{0.3cm} $R,\ R_{\ol{0}},\  R_{\ol{1}}$, $\supp$, $w(\sum a_{\nu}e^{\nu})$;

\ref{Deltalambda} \hspace{1cm} root subsystem;  $\Pi(\Delta')$, $\Delta(\lambda)$,
 $\Pi(\lambda)$, $W(\lambda)$; \hspace{1cm}
\ref{friends} \hspace{0.5cm} friendly;

\ref{Nnu} \hspace{1cm} $\supp(N),\ \Delta(N),\ \Pi(N),\  W(N),\ \ch_{\epsilon},\ \ch$;

\ref{COK} \hspace{1cm}  $\fn^+$, $\CO$, $\CO^{fin}$, $\CO^{inf}$, $N^{\#}$;
 \hspace{1cm}
\ref{ODelta} \hspace{0.3cm} $M(\lambda),\ L(\lambda),\ \CO_{\Delta'}, \ \CO^{fin}_{\Delta'}$;

\ref{lambdalambda'} \hspace{1cm} typical weight;
 \hspace{1cm}
\ref{delta} \hspace{0.5cm} critical weight;
\hspace{1cm}
\ref{equiv} \hspace{0.3cm} $\Lambda$;

\ref{modules} \hspace{1cm} critical module, typical module, $\CO_{crit},\   \CO_{noncrit},\ \CO_{typ},\ \CO_{atyp}$;

\ref{dD} \hspace{1cm} $f^a,\  U',\  \dD_a$;\hspace{1cm}
\ref{En1} \hspace{0.5cm} $\cC_a, \cM(b),\ \cM(b)^e$;

\ref{chain} \hspace{1cm} $\cC^{\alpha},\ \cC^{\alpha_1,\ldots,\alpha_s}$;
\hspace{1cm}
\ref{sigmapr} \hspace{0.5cm} $\Sigma_{pr},\ W_{\Sigma},\Sigma(\lambda)$.

\ref{defnsnowflake} \hspace{1cm} snowflake module;

\ref{prochee}  \hspace{1cm} $M_{pr}(\lambda), \ L_{pr}(\lambda)$, $I(\lambda)$, $I_{pr}(\lambda)$, $w\circ$;

\ref{snow}  \hspace{1cm} $Snow$;  \hspace{1cm} \ref{snowaffi} \hspace{0.3cm}
$\fh^*_k,\ S(k;\Delta')$.

\end{document}